\documentclass{mcom-l}

\usepackage{amsmath} % assumes amsmath package installed
\usepackage{amssymb}  % assumes amsmath package installed
\usepackage{graphics}
\usepackage{epsfig} 
\usepackage{amsthm}
\usepackage{comment}
\usepackage{color}
\usepackage[ruled,vlined]{algorithm2e}

\usepackage[hidelinks]{hyperref}

\usepackage{caption}
\usepackage{subcaption}

\usepackage{booktabs}

\newtheorem{remark}{Remark}[section]
\newtheorem{assumption}{Assumption}[section]
\newtheorem{theorem}{Theorem}[section]
\newtheorem{lemma}{Lemma}[section]
\newtheorem{corollary}{Corollary}[section]
\newtheorem{definition}{Definition}[section]

\definecolor{darkspringgreen}{rgb}{0.09, 0.45, 0.27}

\renewcommand{\chi}{\mathrm{I}}

\DeclareMathOperator*{\argmin}{argmin}
\DeclareMathOperator*{\argmax}{argmax}

\newcommand{\R}{\mathbb R}

\renewcommand{\P}{\mathbb P}

\newcommand{\E}{\mathbb{E}}
\newcommand{\Tr}{\mathrm{tr}}

\def\Id{\mathrm{I}}
\def\e{{\rm u}}

\def\eps{\varepsilon}
\def\bigo{{\mathcal O}}

\newcommand{\cE}{{\mathcal E}}
\newcommand{\D}{\mathbb D}
\newcommand{\Sym}{\mathrm{Sym}}
\def\diag{{\rm diag}}

\def\bigo{\mathcal{O}}
\def\s{\tau}

\def\Dm{D[m]}

\title{Contractivity of neural ODEs:\\ an eigenvalue optimization problem}
\large

\author{Nicola Guglielmi}
\address{Gran Sasso Science Institute, L'Aquila, Italy.}
\email{nicola.guglielmi@gssi.it}

\author{Arturo De Marinis}
\address{Gran Sasso Science Institute, L'Aquila, Italy.}
\email{arturo.demarinis@gssi.it}

\author{Anton Savostianov}
\address{Gran Sasso Science Institute, L'Aquila, Italy; RWTH Aachen University, Germany.}
\email{anton.savostianov@gssi.it}

\author{Francesco Tudisco}
\address{School of Mathematics and Maxwell Institute for Mathematical Sciences, The University of Edinburgh, UK; Gran Sasso Science Institute, L'Aquila, Italy.}
\email{f.tudisco@ed.ac.uk; francesco.tudisco@gssi.it}

\begin{document}
\maketitle

\begin{abstract}

We propose a novel methodology to solve a key eigenvalue optimization problem
which arises in the contractivity analysis of neural ODEs.
When looking at contractivity properties of a one layer weight-tied neural ODE  $\dot{u}(t)=\sigma(Au(t)+b)$ 
(with $u,b \in \R^n$, $A$ is a given $n \times n$ matrix, $\sigma : \R \to \R$ denotes an activation function and for a vector $z \in \R^n$, $\sigma(z) \in \R^n$ has to be interpreted entry-wise), we are led
to study the logarithmic norm of a set of products of type $D A$, where $D$ is a diagonal matrix such that 
$\diag(D) \in \sigma'(\R^n)$.
Specifically, given a real number $c$ (usually $c=0$), the problem consists  in finding the largest positive interval $\chi\subseteq \mathbb [0,\infty)$ such that the logarithmic norm $\mu(DA) \le c$
for all diagonal matrices $D$ with $D_{ii}\in \chi$. 
We propose a two-level nested methodology: an inner level where, for a given $\chi$, we compute an optimizer $D^\star(\chi)$ by a gradient system approach, and an outer level where we tune $\chi$ so that the value $c$ is reached by $\mu(D^\star(\chi)A)$.
We extend the proposed two-level approach to the general multilayer, and possibly time-dependent, case $\dot{u}(t) = \sigma( A_k(t) \ldots \sigma ( A_{1}(t) u(t) + b_{1}(t) ) \ldots  + b_{k}(t) )$ and we propose several numerical examples to illustrate its behaviour, including its stabilizing performance on a one-layer neural ODE applied to the classification of the MNIST handwritten digits dataset.

\end{abstract}

% \begin{keywords}

% \begin{AMS}\end{AMS}

\pagestyle{myheadings}
\thispagestyle{plain}
\markboth{N. Guglielmi, A. De Marinis,  A. Savostianov, F. Tudisco}{Contractivity of neural ODEs}

\section{Introduction}

Consider a ResNet \cite{ResNet},  defined as the sequence of inner layers
(for a given $u_0 \in \R^n$)
\begin{equation} 
u_{k+1}  =  u_{k} + h \sigma \left( A u_k + b \right), \qquad k = 0,\ldots,K-1, \label{eq:ResNet}
\end{equation}
where $u_k, b \in \R^n$, $A \in\R^{n,n}$, $h$ is a (typically small) positive constant, and $\sigma$ is a suitable absolutely continuous and strictly monotonic activation function that acts entry-wise. 
It is immediate to see that \eqref{eq:ResNet} corresponds to the forward Euler discretization of the ordinary 
differential equation \cite{chen2018neural, ruthotto,celledoni}
\begin{equation} 
\dot{u}(t) = \sigma\left( A u(t) + b \right), \qquad t\in[0,T],  
\label{eq:oderesnet}
\end{equation}
with $u_k \approx u(t_k)$.
The robustness of \eqref{eq:oderesnet}  against adversarial attacks \cite{biggio2013evasion,szegedy2013intriguing,goodfellow2014explaining} is one of the key topics of investigation of modern machine learning \cite{carrara2019robustness, yan2019robustness,li2020implicit,carrara2021defending,carrara2022improving,caldelli2022tuning}
and is tightly related to the qualitative properties of the differential equation \eqref{eq:oderesnet}, see e.g. \cite{ruthotto,celledoni,caldelli2022tuning,chang,celledoni2023dynamical,sherry2024designing}. An important stability condition is that small perturbations in the input data do not lead to large differences in the output. More specifically, for any two
solutions $u_1(t)$ and $u_2(t)$ to \eqref{eq:oderesnet}, one requires
\begin{equation} \label{eq:contr}
\| u_1(T) - u_2(T) \| \le C \| u_1(0) - u_2(0) \|,
\end{equation}
for a moderately sized constant $C$.
This topic has received increasing attention in recent years since it has been shown that neural ODEs can significantly benefit in terms of both robustness and accuracy by this type of forward stability \cite{kang2021stable,li2022defending,huang2022adversarial}.
In particular, in the context of deep learning, that is when $h$ is small, it appears clear - at least in the non-stiff case -  that the stability properties of \eqref{eq:ResNet} and \eqref{eq:oderesnet} are strongly related.  

\subsection{Contractivity}

Let us consider the norm $\| \cdot \| = \langle \cdot, \cdot \rangle^{1/2}$ induced by the standard Euclidean 
inner product $\langle \cdot, \cdot \rangle$. With this choice, \eqref{eq:contr} is directly related to the notion of logarithmic norm.
Let $f(t,u) = \sigma\left( A u + b \right)$ and let $\mu \in \R$ such that 
% be the smallest real number such that
\begin{equation} \label{eq:rlc}
\langle f(t,u) - f(t,v) , u - v \rangle \le \mu \| u - v \|^2,
\end{equation}
for all $u, v \in \R^n$ and $t \in [0,T]$.
Then we have \eqref{eq:contr} with $C=\mathrm{e}^{\mu T}$.
By the Lagrange theorem, we get 
(note that since $\sigma$ is assumed to be absolutely continuous it is differentiable almost everywhere)
\[
    \sigma\left(A u + b \right) - \sigma\left(A v +b \right) = 
    \mbox{diag}\Bigl(\sigma' \left(A w  +b \right) \Bigr) A (u-v),
\]
with $w=s u +\left(1-s\right) v$, for a certain $s \in (0,1)$,
and it is immediate to see that the vector field $\sigma\left( A u + b \right)$ - for $\sigma$ absolutely continuous - satisfies
\eqref{eq:rlc} with $\mu=\mu^\star$ given by
\begin{equation} \label{eq:boundnu}
\mu^\star = \sup\limits_{D \in \Omega} \sup\limits_{x \in \R^n, \| x \| = 1} \big\langle D A x, x \big\rangle ,
\end{equation}
where $\Omega$ indicates the set of diagonal matrices $D$ whose 
entries belong to the image $\sigma'(\R)$, i.e. $D_{ii} \in \sigma'(\R)$ for all $i=1,\ldots,n$.

Let $\lambda_{\max}\left( X \right)$ denote the rightmost eigenvalue of $X$ and, 
for a given matrix $B$, 
\[
\Sym(B) = \frac{B+B^\top}{2}.
\]
Then, the logarithmic $2$-norm of $B$ (see e.g. \cite{Sod06}) is defined as
\begin{equation}\nonumber
\max\limits_{u \in \R^n,\ \| u \| = 1} \langle B u , u \rangle = \lambda_{\max} \left( \Sym(B)  \right) =: \mu_2(B).
\end{equation}
In this way we can write \eqref{eq:boundnu} in terms of $\mu_2$  as
\begin{equation} \label{eq:mustar}
\mu^* := \sup\limits_{D \in \Omega}
\mu_2 \left( D A \right). 
\end{equation}
We say that the linear map associated to $B$ is contractive (or strictly contractive) in $L^2$ if 
$\mu_2(B) \le 0$  (or $\mu_2(B) < 0$, respectively), in which case we would get the stability bound
\eqref{eq:contr} with $C=1$ ($C<1$), i.e. for all $T > 0$,
\begin{equation} \nonumber
\| u_1(T) - u_2(T) \| \ \le (<) \ \| u_1(0) - u_2(0) \|,
\end{equation}

Note that, by continuity of eigenvalues and closedness of $\Omega$, the $\sup$ in \eqref{eq:mustar} is a $\max$.
%
%
% \subsection{Our contribution}    \label{sec:ourcont}
%
%
The computation of $\mu^\star$ in \eqref{eq:mustar} is a non-trivial eigenvalue optimization problem, which is also considered 
in \cite{celledoni}. Its computation provides a worst-case contractivity estimate for the neural ODE \eqref{eq:oderesnet}. 

\subsection*{Normalization} 
Note that for an arbitrary activation function $\sigma$ we have 
$
\sigma'(\xi) \in [m_\ell,m_u] \ \mbox{for}  \ \xi \in \R, 
$
with $0 \le m_\ell < m_u$. Moreover, by the positive homogeneity of logarithmic norms we have 
$\mu_2(D A) = m_u \mu_2 \left( \frac{1}{m_u} D A \right)$. 
Thus, we can rescale the problem and consider \eqref{eq:mustar}
over the set $\Omega$ of diagonal matrices $D$ with entries in $[m,1]$, which we indicate by $\Omega_m$, i.e. 
we let
\begin{equation*} \label{eq:Omegam0}
\Omega_m = \{ D \in \D^{n,n} \ : \ m \le D_{ii} \le 1 \quad \forall i \},
\end{equation*}
where $\D^{n,n}$ denotes the set of real $n\times n$ diagonal matrices.

\subsection{Uniform contractivity versus asymptotic stability}

As we have seen we are asking for uniform contractivity with respect to all possible trajectories,
which translates into the uniform logarithmic norm property with respect to all possible diagonal
matrices in $\Omega_m$. 
Why not choose a more local approximation as in \cite{rim2023stability}? Why focus on the entire range? The reason is that this
appears the only way to guarantee the stability bound \eqref{eq:contr}.

Instead of asking for contractivity we may ask for a weaker asymptotic stability condition (as it is done e.g. in \cite{ruthotto}), even
if this would not guarantee a monotone reduction of the norm of the error. However this would be simple for autonomous 
systems but not for non-autonomous problems, as we are going to clarify.

We observe that at time $t$ the linearization of the neural ODE leads us to a linear problem of the form
\[
\dot v(t) = B(t) v(t) \qquad \mbox{where} \ B(t) = D(t) A,
\]
with a certain diagonal matrix $D(t) \in \Omega_m$, that is with entries in $\sigma'(\R)$. 
One may argue that imposing the condition that $\Re \lambda(t) < 0$ for any eigenvalue $\lambda$ of the matrix 
$D(t) A$ and for all $t$ would guarantee asymptotic stability. Unfortunately this is a necessary but not sufficient 
condition (see e.g. the counterexample in \cite[Example 2, page 88]{HNW}).

For this reason, asking that the eigenvalues of $D A$ lie within the left complex half-plane for all $D \in \Omega_m$ 
is not sufficient to guarantee asymptotic stability (and therefore neither contractivity, which is a stronger stability notion).  

\subsection{Overview of our contribution}    \label{sec:ourcont}

In this paper we address several problems connected with contractivity bounds for neural ODEs. 
Most importantly we introduce an algorithm for the efficient computation of $\mu^*$ (see \eqref{eq:boundnu})
both for the constant and for the time-dependent weight matrix $A=A(t)$ (see Problem (P1)).
In contrast, in Problems (P2) and (P3) (which is relevant to the considered application in Section \ref{sec:final_ex}) 
the activation $\sigma$ is considered fixed. In Problem (P3), in particular, the weight matrix $A$ is adjusted to
fulfil a certain stability bound, established a priori. 
In more detail we deal with the following problems.

\subsection*{Problem (P1)} Given a matrix $A$, we aim to compute the smallest value $m^\star$ such that 
\begin{equation} \label{eq:unifcont}
\mu_2\left( D A \right) = 0 \quad \mbox{for all} \ D \in \Omega_{m^\star},
\end{equation}
For the time-dependent case $A=A(t)$, we aim to compute $m^\star(t)$ as
a function of $t$.
Relaxing the bound \eqref{eq:unifcont} to a moderate positive constant is also possible. 
This allows to identify switching functions for which the neural ODE is (uniformly) contractive.
See Section \ref{sec:problem} for the proposed two-level procedure, Section \ref{sec:meth} for the 
details about the eigenvalue optimization method and Section \ref{sec:m} for the computation of the
lower extremal slope $m^*$. A combinatorial relaxation, to the aim of accelerating the computation of critical diagonal matrices 
is proposed in Section \ref{sec:comb}.
% \smallskip

\subsection*{Problem (P2)} Given a matrix $A(t)$ and a given activation function $\sigma$, 
whose slope assumes values in $[m,1]$, we aim to approximate  
\begin{equation} \nonumber
C = \exp \left(  \int\limits_{0}^{T} \mu_2 \left( \widetilde{D}(s) A(s) \right) ds\right) ,
\end{equation}
with 
\begin{equation} \nonumber
\widetilde{D}(t) = \argmax\limits_{D \in \Omega_m} \mu_2\left( D A(t) \right),
\end{equation}
which gives the uniform estimate $\| u_1(T) - u_2(T) \| \le C \| u_1(0) - u_2(0) \|$.
Section \ref{sec:time} explains how to extend the methodology from the case when $A$ has constant coefficients. 
In Subsection \ref{sec:different} we also adopt a different perspective, that is fixing $m$ (instead of optimizing with
respect to $m$) and still optimizing with respect to the set of diagonal matrices $\Omega_m$, compute by quadrature 
a worst-case stability bound for the error, that is a worst case value for the constant in \eqref{eq:contr}. 
\smallskip

Moreover, in view of the application to train a real-life contractive neural ODE (used for image classification), which we will
present in Section \ref{sec:final_ex}, we are interested in the following complementary Problem to (P1) and (P2).
\subsection*{Problem (P3)} 
Given an activation function $\sigma$ whose slope assumes values in $[m,1]$ (with $m$ fixed a priori), and a matrix 
$A$ for which condition \eqref{eq:unifcont} is not fulfilled\footnote{If $m^\star \le m$, then the behaviour of the neural ODE is contractive and there is no need to modify $A$, otherwise there is the option 
to modify $A$ (for example shifting $A$ to $A - \delta \Id$ with $\delta \ge 0$) until the condition above is satisfied.}, we aim to determine a matrix $B$ (which would replace $A$) such that
\begin{equation} \nonumber
\mu_2\left( D B \right) = 0 \quad \mbox{for all} \ D \in \Omega_m,
\end{equation}
with $\| B - A \| \longrightarrow \min $ with $B-A$ belonging to a certain subset of matrices in $\R^{n,n}$. This requires repeatedly solving Problem (P1), while optimizing with respect to $B$. 

In this paper we shall restrict modifications of $A$ to the special form $B - A = - \delta \Id$, with $\delta > 0$ as small as possible
(which we will motivate in Section \ref{sec:final_ex}).

% In order to obtain a sharper result (i.e. a smaller modification), one should deal with
%the following Problem (2).

\subsection*{Extension to multilayer networks}

While we focus on the single-layer case for the sake of simplicity, a variety of observations and results carry over to the multilayer setting as well, where the vector field in the neural ODE is explicitly modeled by several interconnected layers, with weight matrices $A_i$ and bias vectors $b_i$:
\begin{equation} \label{eq:oderesnetm1}\nonumber
\dot{u}(t) = \sigma\biggl( A_k\, \sigma \Bigl( A_{k-1} \ldots \sigma \left( A_{1}\, u(t) + b_{1} \right) 
 \ldots + b_{k-1} \Bigr) + b_{k} \biggr).  
\end{equation}
We will provide details on how to extend to the multilayer case when appropriate throughout the paper (see Subsection \ref{sec:multilayers}).

\subsection{Paper organization}
The paper is organized as follows. Section \ref{sec:problem} describes in detail Problem (P1). %, and how it arises from the study of contractivity of solutions of the neural ODE \eqref{eq:oderesnet}.
Section \ref{sec:meth} outlines our solution strategy, which is a two-level nested approach and is dedicated to the derivation of the inner iteration, while Section \ref{sec:m} is devoted to the numerical solution of the external level. In Subsection \ref{sec:num_ex}, numerical examples on small matrices, for which the solution can be derived analytically, show that the numerical solution is a very reliable approximation of the analytical solution. Section \ref{sec:comb} deals with a combinatorial relaxation of Problem (P1), and provides important insights on its solution. We then deal with Problem (P2) in Section \ref{sec:time}. Finally, Section \ref{sec:final_ex} shows how the proposed algorithm can be employed to train a real-life contractive neural ODE (used for image classification), which is more robust and stable against adversarial attacks than the original one. The last Section \ref{sec:conc} outlines our conclusions. % and future research directions.

\section{A two-level nested iterative method for solving (P1)} \label{sec:problem}
Let us consider a matrix $A \in \R^{n,n}$ with negative logarithmic 2-norm $\mu_2(A)<0$.  

Our goal is to compute
\begin{equation} \label{eq:problem}
m^\star = \argmin\limits_{m \in [0,1]}\,  \{m \ : \  \mu_2(DA) = 0 \quad \text{for all $D \in \Omega_m$} \}.
\end{equation}
%We assume that $\mu_2(A)  < 0.$
Our approach consists of two levels.
\begin{itemize}

\item {\bf Inner problem:\/} 
Given $m>0$, we aim to solve
\begin{equation} \label{eq:probinn}
\min\limits_{D \in\Omega_m} F(D)
\end{equation}
for the functional
\begin{equation}\label{eq:funct}
F (D) = - \mu_2\left( D A \right) = {}-\lambda_{\max} \left( \Sym\left( D A \right) \right),
\end{equation}
with $D\in\Omega_m$.

The functional does not depend explicitely on $m$ but its argument $D \in \Omega_m$ does. 
We indicate the minimum 
\begin{equation*} \label{eq:lambdam}
\min\limits_{D \in \Omega_m} F(D) = -\lambda[m].
\end{equation*}

%The obtained optimizer is denoted by $E(\eps)$ to emphasize its dependence on $\eps$.

\item {\bf Outer problem:\/} We compute the smallest positive value $m \le 1$ such that
\begin{equation} \label{eq:lzero}\nonumber
\lambda[m] = 0,
\end{equation}
and indicate it as $m^\star$.
\end{itemize}

It is obvious that $m^\star$ exists since $\Omega_m$ is compact and non empty.
% (it always contains the identity matrix $\Id$).
\begin{definition} \label{def:extrem}
We call a matrix $\Dm \in \Omega_{m}$
% and minimal entries larger or equal to $m^\star$,
a global extremizer of $F$ if 
\[
\Dm = \arg\min\limits_{D \in \Omega_m} F(D).
\] 
% (equivalently $\mu_2(\Dm A) = \max\limits_{D \in \Omega_m} \mu_2(DA)$).
Similarly, we say that $\Dm$ is a local extremizer of $F$ if there exists a closed
neighborhood $\omega_m \subset \Omega_m$ of $\Dm$ such that
$F(\Dm) \le F(D)$ for all $D \in \omega_m$.
\end{definition}
Then we set $\lambda_{\max} \left( \Sym\left( \Dm A \right) \right) = \lambda[m]$, even in the case
where the minimum is a local one.

\section{Inner problem: a gradient system approach} \label{sec:meth}

%\begin{remark}
%Note that we do not impose norm any constraint on $E$.
%\end{remark}

In order to minimize $F(D) = -\mu_2(DA)$ with respect to $D$ over the set $\Omega_m$,
we let $\lambda$ the rightmost (i.e. largest) eigenvalue of $\Sym(D A)$
over the set $\Omega_m$.
We let $D(\s)$ be a smooth matrix-valued function of the independent variable $\s$ and look for a gradient system for the rightmost eigenvalue of 
$\Sym\left(D(\s) A \right)$.
%
% The considered minimization is an eigenvalue optimization problem.
We propose to solve this problem by integrating a differential equation  with trajectories that follow the gradient descent direction associated to the functional $F$ 
and satisfy further constraints. To develop such a method, we first recall a classical  result (see e.g. 
\cite{Kato2013}) for the derivative of a simple eigenvalue of a differentiable matrix valued function $C(\s)$ with respect to $\s$.
Here we use the notation $\dot{C}(\s):= \frac{d}{d\s} C(\s)$ to denote the derivative with respect to~$\s$.
Moreover, for $A, B \in {\mathbb R}^{n,n}$, we denote (with $\Tr$ standing for the trace)
\[
\langle A,B\rangle = \Tr(B^\top A)
\]
the Frobenius inner product on ${\mathbb R}^{n,n}$. % ($\Tr$ denotes the trace).
\begin{lemma} {\rm \cite[Section II.1.1]{Kato2013}} \label{lem:eigderiv}
Consider a continuously differentiable matrix valued function
$C(\s) :\mathbb \R \to \mathbb \R^{n,n}$, with $C(\s)$ symmetric.
Let $\lambda(\s)$ be a simple eigenvalue of $C(\s)$ for all $\s$ and let $x(\s)$ with $\|x(\s)\|=1$
be the associated (right and left) eigenvector. Then $\lambda (\s)$ is differentiable with
\begin{equation}\nonumber
\dot\lambda(\s) = x(\s)^\top \dot{C}(\s) x(\s) = \Big\langle x(\s) x(\s)^\top, \dot {C}(\s) \Big\rangle.
\label{eq:deigval}
\end{equation}
%FT: lo definiamo subito sopra, non server ripetere a mio avviso
%where we recall that $\langle \cdot, \cdot \rangle$ denotes the Frobenius inner product.
%
\end{lemma}

Let us consider the symmetric part of $D(\s) A$, with $D(\s)$ arbitrary. Applying Lemma \ref{lem:eigderiv} to the rightmost eigenvalue $\lambda(\tau)$ 
of $\Sym(D(\s) A)$ we obtain:
\begin{eqnarray}
\frac{d}{d\tau} \lambda = 
\dot{\lambda} 
& = & \frac12 x^\top \left( \dot{D} A + A^\top \dot{D} \right) x
\nonumber
\\
& = &
\frac12 \Big\langle \dot{D}, A x x^\top \Big\rangle + \frac12 \Big\langle \dot{D} , x x^\top A^\top \Big\rangle
\nonumber
\\
& = & 
\Big\langle \dot{D}, \frac{z x^\top + x z^\top}{2} \Big\rangle = \Big\langle \dot{D}, S \Big\rangle,
\qquad \mbox{with} \quad S = \Sym\left( z x^\top \right),
\nonumber
\label{eq:derlam}
\end{eqnarray}
where $z = A x$ and  the dependence on $\tau$ %time dependence 
is omitted for conciseness.
% This identifies the unconstrained gradient of $F$ as 
% \begin{equation} \label{eq:S}
% S = {}-\Sym\left( z x^\top \right).
% \end{equation}

To comply with the constraint $D \in \D^{n,n}$, in view of Lemma~\ref{lem:eigderiv} we are 
thus led to the following constrained optimization problem for the admissible direction of 
steepest descent.
In the following, for a given matrix $B$, $\diag(B)$ denotes the diagonal matrix having 
diagonal entries equal to those of $B$.

% Imposing the diagonal structure yields the following constrained optimization problem.

\begin{lemma} % [Direction of steepest admissible descent with no inequality constraints]
\label{lem:opt} 
Let $S \in \R^{n,n}$ a matrix with at least a nonzero diagonal entry. A solution of the optimization problem
\begin{eqnarray}
Z_\star  & = & \argmin_{Z \in \D^{n,n},\ \|Z\|_F=1 } \ \langle  Z, S \rangle \nonumber
\label{eq:opt}
\end{eqnarray}
(where the norm constraint $\|Z\|_F=1$ is only considered to get a unique solution) is given by % (see \eqref{eq:S})
\begin{align}
\nu Z_\star & =  - \diag(S)
\label{eq:G}
\end{align}
%where (see \eqref{eq:S})
%\begin{equation}
%G = \diag(S),
%\label{eq:G}
%\end{equation}
and $\nu$ is the Frobenius norm of the matrix on the right-hand side. 
% The solution exists if $G$ is non zero (which we are able to prove).
\end{lemma}
\begin{proof}
The proof is obtained by simply noting that \eqref{eq:G} represents the orthogonal projection 
-- with respect to the Frobenius inner product -- % of the anti-gradient onto the subspace 
onto the set of diagonal matrices. 
% associated to $F$ with the considered constraints.
\end{proof}

Observing that for $S = \Sym\left( z x^\top \right)$, we have 
$\diag\left( \Sym\left( z x^\top \right) \right) = \diag\left( z x^\top \right)$,
using Lemma \ref{lem:opt}, we obtain the constrained gradient for the functional 
\eqref{eq:funct} % is given by 
\begin{equation} \label{eq:grad}
G = - \diag\left(  z x^\top \right). 
\end{equation}
Hence the
constrained gradient system to minimize $F(D)$ is given by
\begin{equation} \label{eq:Ddot}
\dot D = -G = \diag\left(  z x^\top \right),
\end{equation}
which holds until one of the entries of $D(\s)$ equals  $m$ or  $1$,
in which case inequality constraints should be considered.
%
% \subsection{The system of ODEs}
%
Note that the ODE system \eqref{eq:Ddot} can be conveniently rewritten componentwise.
With $D = \diag\left( d_1, d_2, \ldots,d_n \right)$, we get
\begin{equation} \label{eq:Ddotcw} \nonumber
\dot d_{i} = x_i z_i,
\end{equation}
% where 
% \begin{equation} \nonumber
%= \left( \begin{array}{ccccc}
%d_1 & 0 & \ldots & 0 & 0 \\
%0 & d_2 & 0 & \ldots & 0 \\
%\vdots & \ldots & \ddots & \ldots & 0 \\
%0 & \ldots &  0 & d_{n-1} & 0 \\
%0 & \ldots & 0 & 0 & d_{n}
%\end{array} \right),
% \end{equation}
with $(\lambda,x)$ the rightmost eigenpair of $\Sym\left( D A \right)$,
% $x$ the associated eigenvector to $\lambda$ 
and $z=A x$.

\subsection{Inequality constraints: admissible directions}\label{sec:addir}

To comply with the lower bound constraint of the diagonal matrix $D(\s)$,  we need to impose that $\dot{D}_{ii}(\s) \ge 0$ for all $i$ 
such that $D_{ii}(\s) = m$.
For a diagonal matrix $D \in \Omega_m$ we set 
\begin{equation}\nonumber
\cE_0(D) := \{i :\, D_{ii} = m \},
\label{eq:Dzero}
\end{equation}
and we define $P_{\cE_0(D)}$ as follows: 
% the orthogonal projection from $\D^{n,n}$ onto the shifted space $\D_+^{n,n}$ given by
\begin{equation}\nonumber
\left( P_{\cE_0(D)} Z \right)_{ii} = 
\left\{ \begin{array}{lr} Z_{ii} & \mbox{if} \ D_{ii} > m \\[2mm] \max\left( 0, Z_{ii} \right) & \mbox{if} \ D_{ii} = m
\end{array} \right. \, .
\label{eq:proj}  
\end{equation}
Similarly, we proceed to impose the upper bound constraint: we set 
\begin{equation}\nonumber
\cE_1(D) := \{i :\, D_{ii} = 1 \},
\label{eq:Done}
\end{equation}
and we define $P_{\cE_1(D)}$ as 
% the orthogonal projection from $\D^{n,n}$ onto the shifted space $\D_{-}^{n,n}$ given by
\begin{equation}\nonumber
\left( P_{\cE_1(D)} Z \right)_{ii} = 
\left\{ \begin{array}{lr} Z_{ii} & \mbox{if} \ D_{ii} < 1 \\[2mm] \min\left( 0, Z_{ii} \right) & \mbox{if} \ D_{ii} = 1
\end{array} \right.\, .
\label{eq:proju}  
\end{equation}
%In what we will also denote and denote in short $\Pact_D =P_{\cE_0(D)}$.
%which we denote in short as $\Pactu_D =P_{\cE_1(D)}$.

\subsection{Constrained gradient system}
\label{sec:KKT}
To include constraints in the gradient system we make use of Karush-Kuhn-Tucker (KKT) conditions. 
Let us consider, for conciseness, only  the inequality $D_{ii} \ge m$ for all $i$. The  
optimization problem to determine the constrained steepest descent direction is
\begin{equation} \label{opt_problem}
Z_\star  = \argmin_{Z \in \D^{n,n},\,\|Z\|_F=1,  D \ge m} \ \langle  S,  Z \rangle.
\end{equation}
Analogously to Lemma \ref{lem:opt} we have that the solution of \eqref{opt_problem} satisfies the KKT 
conditions %(see, e.g., \cite[Ch. 2.2]{geigerkanzow})
\begin{align*}\label{KKT}
	&%\kappa 
	Z = -\widehat G,  \\
	&\nu_{i} Z_{ii} = 0 \text{ for all } i \in \mathcal{E}_0(D),  \\
	&\nu_{i}      \ge 0 \text{ for all } i \in \mathcal{E}_0(D) \,,
\end{align*}
where 
\[
\upsilon \widehat G = \diag(S) + \sum_{i \in \cE_0(D)} \nu_{i} u_i u_i^\top,
\] 
$u_\ell$ is the $\ell$-th canonical versor ($(u_\ell)_i=1$ if $\ell=i$ and $(u_\ell)_i=0$ otherwise), while 
$\{ \nu_{i} \}$ are the Lagrange multipliers associated to the non-negativity constraints,
and $\upsilon$ is the Frobenius norm of the matrix on the right-hand side.  In a completely analogous way, we can treat the inequality $D_{ii} \le 1$.

\begin{comment}
\subsubsection{Penalization}

An alternative is penalization. Introducing the functions
\[
(a)_+ = \max{(a,0)} \qquad \mbox{and} \qquad
(a)_- = \min{(a,0)},
\]
we define for $D$
\begin{equation}
Q_+(D) = \frac12 \sum\limits_{i=1}^{n} \left(D_{ii} - m \right)_{-}^2, \qquad
Q_-(D) = \frac12 \sum\limits_{i=1}^{n} \left(D_{ii} - 1 \right)_{+}^2,
\end{equation}
and the penalty function $Q(D) = Q_+(D) + Q_-(D)$.

Then we can minimize the functional
\begin{equation}
F(D) = {}-\lambda + c Q(D),
\end{equation}
with a large $c > 0$ when the inequality constraints activate.
\end{comment}

\subsection{Structure of extremizers: a theoretical result}

We state now an important theoretical result on the structure of extremizers (see Definition \ref{def:extrem}).

\begin{theorem}[necessary condition for an extremizer to assume intermediate values in $(m,1)$] \label{th:stru}
For a given matrix $A \in \R^{n,n}$ such that $\mu_2(A) < 0$, 
assume that $\Dm$ is a local extremizer % (see Definition \ref{def:extrem}) 
of the inner optimization Problem \eqref{eq:probinn}, 
and assume (generically) that the rightmost eigenvalue $\lambda$ of $B = \Sym(\Dm A)$ 
is simple.
Let $x$ be the eigenvector associated to $\lambda$, 
$z=Ax$ and $x_i$ and $z_i$ are the $i$-th entries of the vectors $x$ and $z$, respectively.   

Then the $i$-th diagonal entry of  $\Dm$ either assumes values in 
$\{ m,1 \}$ (extremal ones) or - if $m < d_{ii} < 1$ - the following condition has to hold true:
\begin{equation} \label{eq:condex}
x_i z_i = 0.
\end{equation}
\end{theorem}
\begin{proof}
Assume that the $i$-th diagonal entry of the extremizer $D = \Dm$ is equal to $d$, with
$m < d < 1$. Now consider a small perturbation $\delta$ of such a diagonal element and consider
the associated diagonal matrix
\[
D(\delta) = D + \delta u_i u_i^\top \in \Omega_m, 
\]
(with $u_i$ the $i$-th versor of the canonical basis of $\R^n$). 
Note that $D(\delta) \in \Omega_{m}$ is still admissible for $|\delta|$ sufficiently small.

We indicate by $\lambda = \lambda[m]$ the rightmost eigenvalue of $\Sym\left(D A \right)$ and by
$\lambda(\delta)$ the perturbed rightmost eigenvalue of $\Sym\left(D(\delta) A \right)$.
By applying the first order formula for eigenvalues we obtain, 
for $\lambda(\delta)$ the first order expansion: 
\begin{equation} \nonumber
% \mu_2 \left( D A \right) = \mu_2 \left( \Dm A \right) + 
\lambda(\delta) = \lambda + \frac12 \delta x^\top \left(  u_i u_i^\top A + A^\top u_i u_i^\top \right) x + \bigo(\delta^2)
= \lambda + \delta x_i z_i + \bigo(\delta^2),
\end{equation}
with $x_i$ the $i$-th component of the eigenvector $x$ and $z_i$ the $i$-th component of $z = A x$.
If $x_i z_i \neq 0$,
% were nonzero, 
suitably choosing the sign of $\delta$ would increase $\lambda(\delta)$ for $\delta$ sufficiently small,
and thus would imply $\mu_2(D(\delta) A) > \mu_2(D A)$, that is a contradiction. 
\end{proof}

We are also in the position to state the following results.
\begin{lemma} \label{lem:m}
Assume that $\Dm \in \Omega_{m}$ is an extremizer s.t. $F \left(\Dm \right) > 0$. Then 
$\min\limits_{1 \le i \le n} D_{ii} = m$.
\end{lemma}
\begin{proof}
Let $\lambda = \lambda[m] < 0$ the rightmost eigenvalue of $\Sym\left( \Dm A \right)$.
% It is easy to prove that any extremizer has necessarily at least one entry which assumes the extremal value $m$. 
Assume that the statement is not true, that is
\begin{equation}\nonumber
\min\limits_{1 \le i \le n} D_{ii} = \widehat{m} > m.
\end{equation}
Then, there exists a positive $\gamma < 1$ such that $\widehat D = \gamma \Dm$ has still entries in $[m,1]$, and thus
belongs to $\Omega_m$,
% \asnote{this is a problem because we defined $\Omega_m$ such that $\| D \| = 1$; we probably should fix the definition?}, 
i.e. it is admissible.
This implies that $\Sym(\widehat D A)$ has the rightmost eigenvalue $\widehat{\lambda} = \gamma \lambda > \lambda$
(which here is negative), which contradicts extremality.
\end{proof}
\begin{comment}
Based on Lemma \ref{lem:1}, until the functional $F$ is positive (that is the interesting case), we can change \eqref{eq:Omegam0} into
\begin{equation}
\Omega_m = \{ D \in \D^{n,n} \ : \ m \le D_{ii} \le 1 \quad \forall i, \quad \exists j \ : \ D_{jj} = m \}.
\end{equation}
\end{comment}

Similarly we can prove the following analogous result.
\begin{lemma} \label{lem:1}
Assume that $\Dm \in \Omega_{m}$ is an extremizer s.t.  $F \left(\Dm \right) < 0$.  Then 
$\max\limits_{1 \le i \le n} D_{ii} = 1$.
\end{lemma}
\begin{comment}
In such case we can also include in the definition of $\Omega_m$ the fact that at least one entry of $D$ should be equal to $1$.
\end{comment}

%\begin{remark} \rm
%Notice that, if $F \left(\Dm \right)=0$, then Lemma \ref{lem:m} and Lemma \ref{lem:1} do not hold anymore. They hold only when the strict inequalities are satisfied, because if the functional $F=0$, then the contradiction does not hold anymore. However..
%\end{remark}
%\medskip
%
%\asnote{But can't we do some sort of continuity argument here? Say, $F \left(\Dm \right)=0$ and $F_{m+\Delta m} \left(D[m +\Delta m] \right)>0$ for all positive $\Delta m$s?}
%\medskip
%\ngnote{YES, I AGREE, AT LEAST PARTIALLY. Then I would give the following result}
%\medskip
%\begin{NIC}
Finally we can prove the following result.
\begin{lemma} \label{lem:m1}
Assume that $\Dm \in \Omega_{m}$ is an extremizer s.t.  $F \left(\Dm \right) = 0$
and assume (generically) that $\Dm$ (as a function of $m$)  is continuous at $m$.  Then 
$\max\limits_{1 \le i \le n} D_{ii} = 1$ and 
$\min\limits_{1 \le i \le n} D_{ii} = m$.
\end{lemma}
\begin{proof}
The proof combines Lemma \ref{lem:m} and \ref{lem:1} and continuity of eigenvalues.
\end{proof}
    
% \end{NIC}
%%
%%%Assume, by contradiction, that $D_{ii} > m$ for all $i$. Then, by previous Theorem \ref{th:stru},
%%%we have that $x_i z_i = 0$ for all $i$ such that $D_{ii} \neq 1$.   
%%
%%Let $i$ indicate the minimal entry $m'$ of $\Dm$, which is smaller than $1$ (i.e. $1 > m' > m$). This would
%%imply that 
%%\[
%%\min\limits_{D \in \Omega_m'} F'(D) =
%%\min\limits_{D \in \Omega_m} F(D)
%%\]
%%
%%Let $\delta=\frac{m'-m}{2}$. Then construct the admissible matrix
%%\[
%%D' = (\Dm + \delta \Id) \frac{1}{1+\delta} 
%%\]
%%Then 
%%\[
%%\mu_2 \left( \Sym\left(D' A \right) \right) = \frac{1}{1+\delta}  
%%\]
%%
%%%%using the same argument as before,
%%%Let $i$ indicate the minimal entry of $\Dm$ which is smaller than $1$ (and  larger than $m$).
%%%Then consider again the matrix
%%%\[
%%%D(\delta) = \Dm + \delta u_i u_i^\top \in \Omega_m, 
%%%\] 
%%%which is such that $F \left( D(\delta) \right) \ge F \left( \Dm \right)$ for $d - m < \delta < 0$.  
%%%
%%%Assume that no such value exists. This would imply that $D=\Id$. But if $\mu(A) < 0$, this is not possible by continuity
%%%of eigenvalues.  

\begin{remark}[Multiple eigenvalues] \rm 
% \label{rem:mult-eig}
Throughout the article we assume that the rightmost eigenvalue is simple.
However, along a trajectory $D(t)$, a multiple eigenvalue $\lambda(t)$ may occur at some finite $t$ because of a coalescence of eigenvalues. Even if some continuous trajectory runs into a coalescence, this generally does not happen after discretization of the differential equation, 
and so the computation will not be affected.
As such situation is either non-generic or can happen generically only at isolated times $t$, it does not affect the computation after discretization 
of the differential equation.

However, close-to-multiple eigenvalues may impair the accuracy of the computed eigenpairs and should be detected numerically and treated carefully.
\end{remark}

\subsection{Numerical integration}

In Algorithm \ref{alg_prEul} we provide a schematic description of the 
integration step.
For simplicity, we make use of the Euler method, but considering other 
explicit methods would be similar.
Classical error control is not necessary here since the goal is just 
that of diminishing the functional $F(D)$ (see lines {\bf 9} and {\bf 10} 
in Algorithm \ref{alg_prEul}).

\medskip
\begin{algorithm}[t] 
\DontPrintSemicolon
% \Dm
\KwData{$A, m, \theta > 1, D_k \approx D(t_k)$, $f_k = F(D_k)$, $h_{k}$ (proposed step size)}
\KwResult{$D_{k+1}$, $h_{k+1}$}
\Begin{
\nl Initialize the step size by the proposed one, $h=h_{k}$\; 
\nl Compute the eigenvector 
$x_k$ of $\Sym(D_k A)$ to the rightmost eigenvalue $\lambda_k$ such that $\| x_k \| = 1$\; 
\nl Compute $z_k = A x_k$ \;
\nl Compute the gradient $G_k$ according to \eqref{eq:grad}\;
\nl Compute $\dot{D}_k$ according to \eqref{eq:Ddot} \;
\nl If inequality constraints are active, suitably modify $\dot{D}_k$ as in \ref{sec:addir}\;
\nl Initialize $f(h) = f_k$\;
\nl Let $D(h) = D_k + h \dot{D}_k$\;
\While{$f(h) \ge f_k$}{
% \nl Let $D(h) = \Id + D(h)$\;
\nl Compute rightmost eigenvalue $\lambda(h)$ of $\Sym(D(h) A)$ and let $f(h) = -\lambda(h)$\;
\If{$f(h) \ge f_k$}{Reduce the step size, $h:=h/\theta$\; \nl Set $\mathrm{reject=1}$}
}
%\nl Initialize $h_{\rm next}=h$\;
\eIf{$\mathrm{reject=0}$}
{Set $h_{\rm next} := \theta h$}
{Set $h_{\rm next} := h$}
\nl Set $h_{k+1}=h_{\rm next}$, $\lambda_{k+1}= \lambda(h)$, and the starting value for the next step as 
$D_{k+1}=D(h)$\;
\Return
}
\caption{Integration step for the constrained gradient system (with $m$ fixed)
\label{alg_prEul}}
\end{algorithm}

% \subsection*{Numerical considerations on inequality constraints}

%\textcolor{red}{Is this what we describe in Section "admissible directions"?}
Also note that, in practice,  it is not necessary to compute the Lagrange multipliers for the entries of $D$ that assume extremal values as in Section \ref{sec:KKT}; the numerical procedure has just
to check when one of the entries of the diagonal matrix $D$ reaches either the extreme value $m$ or the
extreme value $1$ (by an event-detection algorithm) coupled with the numerical integrator. 
In the first case, as in Section \ref{sec:addir}, if $D_{ii}=m$, we have to check the steepest descent direction 
$-G$, and set $\dot{D}_{ii} = 0$ if $-G_{ii} < 0$ or let it unvaried otherwise.
In the second case, if the entry $D_{ii}=1$, we have to set $\dot{D}_{ii} = 0$ if $-G_{ii} > 0$ and let it unvaried otherwise.

\begin{comment}   
\begin{remark}	\rm
We do not know whether an entry may bounce after reaching one of the boundary values
($m$ or $1$). 
If we knew it cannot, we could freeze such an entry as soon as it assumes one of the
extremal values and exclude it from the gradient system. 
%If instead we observe that it can bounce we should consider the constraints in the 
%expression of the gradient or penalize.
\end{remark}
\end{comment}

\subsection{Multilayer networks} \label{sec:multilayers}

Let us examine now networks with several interconnected layers; for simplicity we consider 
time-independent $A_i$ and $b_i$:

\begin{equation} \label{eq:oderesnetm}\nonumber
\dot{u}(t) = \sigma\biggl( A_k\, \sigma \Bigl( A_{k-1} \ldots \sigma \left( A_{1}\, u(t) + b_{1} \right) 
 \ldots + b_{k-1} \Bigr) + b_{k} \biggr).  
\end{equation}
Making use of Lagrange theorem, we get
\begin{eqnarray*}
&&    \sigma\biggl( A_k\, \sigma \Bigl( A_{k-1} \ldots \sigma \left( A_{1}\, u(t) + b_{1} \right) 
 \ldots + b_{k-1} \Bigr) + b_{k} \biggr) \\
&& -  \sigma\biggl( A_k\, \sigma \Bigl( A_{k-1} \ldots \sigma \left( A_{1}\, v(t) + b_{1} \right) 
 \ldots + b_{k-1} \Bigr) + b_{k} \biggr) = \\
&&
    \mbox{diag}\Bigl(\sigma' (y_k) \Bigr) A_k   \mbox{diag}\Bigl(\sigma' (y_{k-1}) \Bigr) A_{k-1} 
    \cdots  \mbox{diag}\Bigl(\sigma' (y_1) \Bigr) A_1 (u-v),
\end{eqnarray*}
for suitable vectors $y_1, \ldots, y_{k-1}, y_k$.
The proposed methodology extends to the study of the logarithmic norm
of matrices of the form  
\begin{equation}\nonumber
\Sym\left( D_k A_k D_{k-1} A_{k-1} \cdots D_1 A_1 \right).
\end{equation} 
Again the functional we aim to minimize is
\[
F (D_1,\ldots,D_k) = {}-\lambda_{\max} \left( \Sym\left( D_k A_k D_{k-1} A_{k-1} \cdots D_1 A_1 \right) \right),
\]
for $D_i \in \Omega_m$ for $i=1,\ldots,k$.
Letting 
\[
\left( D_1[m], \ldots, D_k[m] \right) = \argmin\limits_{D_1,\ldots, D_k \in \Omega_m} F(D_1,\ldots,D_k),
\]
we indicate by $\lambda[m]$ the rightmost eigenvalue of $\Sym\left( D_k A_k D_{k-1} A_{k-1} \cdots D_1 A_1 \right)$.

Optimizing with respect to $m$, our task is again to find the smallest solution of the scalar equation
$
\lambda[m] = 0,
$
which we still indicate as $m^\star$.

\subsection*{Inner iteration}
%This still implies, for two different solutions associated with different initial conditions, 
%$\| u_2(T) - u_1(T) \|_2 \le \| u_1(0) - u_2(0) \|_2 $.
% , with $C=\e^{\mu^* T}$.

In analogy to the single layer case,
let us consider the symmetric part of $D_k(\s) A_k \cdots D_1(\s) A_1$, with $D_i(\s)$ diagonal for $i=1,\ldots,k$. 
Applying Lemma \ref{lem:eigderiv} to the rightmost eigenvalue $\lambda(\tau)$ 
of $\Sym(D_k(\s) A_k \cdots D_1(\s) A_1)$ we obtain (omitting the dependence on $\tau$):
\begin{eqnarray}
\nonumber
\frac{d}{d\tau} \lambda = 
\dot{\lambda} 
& = & \frac12 x^\top \left( \Sym \left( \dot{D_k} A_k D_{k-1} A_{k-1} \cdots D_1 A_1 +  \right. \right.
\\
& & 
\left. \left. D_k A_k \dot{D}_{k-1} A_{k-1} \cdots D_1 A_1 + \cdots + 
              D_k A_k D_{k-1} A_{k-1} \cdots \dot{D}_1 A_1 \right) \right) x
\nonumber
\\
% & = & \ldots
% \\
% & = &
% \frac12 \Big\langle \dot{D}, A x x^\top \Big\rangle + \frac12 \Big\langle \dot{D} , x x^\top A \Big\rangle
% \nonumber
% \\
& = & 
\sum\limits_{i=1}^{k} \Big\langle \dot{D}_i, \frac{z_i w_i^\top + w_i z_i^\top}{2} \Big\rangle = 
\sum\limits_{i=1}^{k} \Big\langle \dot{D}_i, \Sym\left( z_i w_i^\top \right) \Big\rangle,\nonumber
\label{eq:derlamk}
\end{eqnarray}
with $z_1 = A_1 x$, \ldots, $z_{k-1} = A_{k-1} \cdots D_1 A_1 x$, $z_k=A_kD_{k-1}A_{k-1}\cdots D_1A_1x$ and
     $w_1 = A_2^\top D_2\cdots A_{k-1}^\top D_{k-1}A_k^\top D_k x$, \ldots, $w_{k-1} = A_k^\top D_k x, \ldots, w_k = x$.

This allows us to get an expression for the free gradient analogous to the single layer case.     
  
The associated constrained gradient system is given --- in analogy to \eqref{eq:Ddot} --- by
\begin{equation} \label{eq:MDdot}
\dot D_i % =  G_i 
= \diag\left( z_i w_i^\top  \right), \qquad i=1,\ldots,k.
\end{equation}
Every ODE in \eqref{eq:MDdot} is integrated with no modifications until one of the entries of $D_i(\s)$ equals  $m$ or  $1$,
in which case inequality constraints should be considered.

\section{Outer iteration: computing $m^\star$} \label{sec:m}

In order to compute  the smallest positive value $m^\star$ such that $\lambda[m^*]=0$ in the outer-level iteration, we have to face two issues:
\begin{itemize}
\item[(i) ] which initial value of $m$ to consider;
\item[(ii) ] how to tune $m$ until the desired approximation of  $m^\star$.
\end{itemize}

    In this section we discuss a number of results that will directly lead to a proposal for a possible way to address both these two issues. Before - in the next two subsections - we provide some insights into the problem.

\subsection{The value $\mu_2(A)$ is strongly unspecific.} % We need more information on $A$}

Unfortunately, the only information of $\mu_2(A)$ is very unspecific for the determination of $m^\star$
and the behaviour of $\Sym(D A)$. Here is an illuminating example, just in dimension $2$. 
Let $a_{11}, a_{22} < 0$, $k > 0$ and 
\begin{equation} \nonumber
A = \left(
\begin{array}{cc}
 a_{11} & k \\
-k      & a_{22} \\
\end{array}
\right) \qquad \mbox{and} \qquad D = \diag(1,m).
\end{equation}
We have $\Sym(A) = \diag\left( a_{11}, a_{22} \right)$ and  $\mu_2(A) = \max\{ a_{11}, a_{22} \} < 0$;  we get
\begin{equation*}
\Sym(DA) = \frac12 \left(
\begin{array}{cc}
 2 a_{11} & k-k m \\
 k-k m & 2 a_{22} m \\
\end{array}
\right),
\end{equation*}
having eigenvalues
\begin{eqnarray*}
\lambda_{1,2} & = & \frac{1}{2} \left(\mp \sqrt{a_{11}^2-2 a_{11}
   a_{22} m+a_{22}^2 m^2+m^2 k^2-2 m
   k^2+k^2}+a_{11}+a_{22} m\right),
%\\
%\lambda_2 & = & \frac{1}{2}
   %\left(\sqrt{a_{11}^2-2 a_{11} a_{22}
   %d+a_{22}^2 d^2+d^2 k^2-2 d k^2+k^2}+a_{11}+a_{22}
   %d\right)
\end{eqnarray*}
We indicate by $m^\star(k)$ the value of $m^\star$ as a function of $k$.
Choosing $k=0$ ($A$ diagonal) we get $m^\star(0)=0$.
Choosing instead $k \gg 1$, we get the first order expansion
\begin{eqnarray*}
\lambda_{1,2} & = & \pm\,\frac{1}{2} (1-m) k+\frac{1}{2} (a_{11}+a_{22}
   m)+\bigo\left(\frac{1}{k}\right).
%\\
%\lambda_2 & = & {}-\frac{1}{2} (1-d)
   %k+\frac{1}{2} (a_{11}+a_{22}
   %d)+\bigo\left(\frac{1}{k}\right).
\end{eqnarray*}
This immediately implies that $\lim\limits_{k \rightarrow \infty} m^\star(k) = 1$,
which is consistent with % (suggested by) 
Theorem \ref{th:ubmstar}, as $\| A \|_2 = k + (a_{11} - a_{22})/2 + \bigo(1/k).$

This shows that, in order to obtain a meaningful estimate of $m^\star$, we need more information on $A$ than the only logarithmic norm. For example, the non-normality of the matrix and the structure of the non-symmetric part of~$A$.

\subsection{A theoretical upper bound for $m^\star$}

We have the following upper bound for $m^\star$, obtained for a general absolute norm by applying the Bauer--Fike theorem.

Recall that a general property of any logarithmic norm states that for any matrix $A$, 
it holds $-\| A \| \le \mu(A) \le \| A \|$.

\begin{theorem}[Upper bound for $m^\star$] \label{th:ubmstar}
For a given matrix $A \in \R^{n,n}$ such that $\mu(A) < 0$, %%% state for $\mu_2$ ?
assume that $m^\star$ solves Problem \eqref{eq:problem}. Let $\| \cdot \|$ be an absolute norm
and 
\begin{equation} \label{eq:mstarub2} \nonumber
\beta(A) = \frac{\| A \| + \| A^\top \|}{2}.
\end{equation}
Then, if $|\mu(A)| < \beta(A)$,
\begin{equation} \label{eq:mstarub}
m^\star \le 1 - \frac{|\mu(A)|}{\beta(A)} := m^\star_{ub} < 1.
\end{equation}
If  $\| \cdot \|$ is such that $\| A \| = \| A^\top \|$, then 
 $|\mu(A)| < \|A\|$   
 %$|\mu(A)| < \beta(A)$ 
and $\displaystyle m^\star \le 1 - \frac{|\mu(A)|}{\| A \|}.$
\end{theorem}
\begin{proof}
Let $D = D[m^\star] \in \Omega_{m^\star}$ be an extremizer, % with $\min\limits_{1 \le i \le n} D_{ii} = m^\star$ 
and write $D= \Id + E$.
% We denote an extremizer $D \in \Omega_{m^\star}$ with $\min\limits_{1 \le i \le n} D_{ii} = m^\star$. We have
Then we have $\mu(DA)=0$ and
\begin{equation*}
\Sym(DA) = \frac12 \left( D A + A^\top D \right) = \Sym(A) + \frac12 \left( E A + A^\top E \right).
\end{equation*}
Applying Bauer-Fike theorem and exploiting the normality of $\Sym(DA)$, we obtain
\begin{equation*}
|\mu(A)| \le \| \Sym(DA) - \Sym(A) \| = \frac12 \| E A + A^\top E \| \le \frac{ \| A \| + \| A^\top \| }{2} \| E \| = \beta (A) \|E\|.
\end{equation*}
Using Lemma \ref{lem:1} and 
the assumption that $\| \cdot \|$ is an absolute norm we get that, if the largest entry of $D$ equals $1$,
$\| E \| = 1-m^\star$, otherwise $\| E \| \le 1-m^\star$.
Consequently - since $\beta(A) \ge |\mu(A)|$ - we get
\begin{equation*}
|\mu(A)| \le \beta(A) (1-m^\star),
\end{equation*}
which concludes the proof.

\end{proof}

\begin{corollary}
Under the assumptions of Theorem {\rm \ref{th:ubmstar}},
%if $A$ is normal then
%\begin{equation}
%m^\star \le 1 - \frac{|\mu(A)|}{\rho(A)}.
%\end{equation}
%where $\rho(A)$ is the spectral radius of $A$.
%If - in addition - $A$ is symmetric,
if $A$ is symmetric and $\| \cdot \| = \| \cdot \|_2$, then
\begin{equation}\nonumber
m^\star \le 1 - \kappa_2(A)^{-1}, \qquad \mbox{where} \quad \kappa_2(A) = \bigg|\frac{\lambda_{\max}(A)}{\lambda_{\min}(A)}\bigg|
\end{equation}
is the condition number of $A$ ($\lambda_{\max}(A)=\mu_2(A)$ denotes the smallest eigenvalue (in modulus) of $A$ 
and $\lambda_{\min}(A)=-\rho(A)$ the rightmost eigenvalue of $A$ (in modulus)). 
\end{corollary}
\begin{proof}
For a symmetric matrix $\| A \|_2 = \rho(A)$; moreover by the contractivity assumption, 
$\mu_2(A) = \lambda_{\max}(A)$ and $\rho(A) = -\lambda_{\min} (A)$.
\end{proof}

\subsection{Iterative computation of $m^\star$} \label{sec:outer}

We aim to compute $m^\star$, the smallest positive solution of the one-dimensional root-finding problem 
\begin{equation} \label{eq:zero}
F_{m} (D[m]) = 0,
\end{equation}
where $D[m]$ denotes a minimizer of $F(D)$.
This can be solved by a variety of methods, such as bisection.
Theorem \ref{th:ubmstar} suggests to start applying the outer level step with an initial value 
$m \le m^\star_{ub}$ (see \eqref{eq:mstarub}).

After computing an extremizer  $D[m]$, which minimizes the functional
$F_{m}(D)$ with respect to $D$, for fixed $m$, in the case when $|F_{m}(D[m])|$ is not below a certain tolerance, we have to further 
modify $m$,  until we reach a value $\tilde m$ such that $F_{\tilde m} (D[\tilde m]) = 0$, which would approximate the searched value $m^\star=\tilde m$.  
% In order to find such a $\tilde m$, we need to solve the root-finding problem $F_{m} (D[m]) = 0$. 
We derive now a costless variational formula and a fast outer iteration to approximate $m^\star$. 
With such formula, we aim for a locally quadratically convergent Newton-type 
method, which can be justified under regularity assumptions that appear to be usually satisfied. If these assumptions
are not met, we can always resort to bisection. The algorithm proposed  in the following  in fact uses a combined 
Newton-bisection approach.

Let $D = D[m]$ be an extremizer. 
We denote by $J=J[m]$ the diagonal matrix with binary entries (either $0$ or $1$),
according to the following: 
\begin{equation} \label{eq:Jk}
J_{kk} = \left\{ \begin{array}{rl} 1 & \mbox{if} \ D_{kk} = m, \\[2mm] 0 & \mbox{otherwise.} \end{array} \right.
\end{equation} 
Next we make an assumption, which we have always encountered in our experiments.
\begin{assumption} \label{ass:D-m}
For $m$ sufficiently close to $m^\star$, we assume for the extremizer $D[m]$ that:
\begin{itemize}
\item the eigenvalue $\lambda[m]=\lambda_{\max}\left( \Sym(D[m] A) \right)$ is a simple eigenvalue;
\item the map $m \mapsto D[m]$ is continuously differentiable;
\item the matrix $J[m]$ is constant\footnote{This means that the entries that are equal to the extremal value $m$ in the matrix $D[m]$ remain extremal 
for $m$ in a neighborhood of $m^\star$, which looks quite reasonable.}.
\end{itemize}
\end{assumption}
Under this assumption, the branch of eigenvalues $\lambda[m]$ and its corresponding eigenvectors $x[m]$ 
are also continuously differentiable functions of $m$ in a neighbourhood of $m^\star$. 
The following result gives us an explicit and easily computable expression for the derivative of 
$\phi[m]= F(D[m]) = -\lambda[m]$ with respect to $m$ in simple terms. Its proof exploits 
Theorem \ref{th:stru}.
\begin{theorem}[Derivative of $\phi$] 
\label{thm:phi-derivative}
Under Assumption~{\rm \ref{ass:D-m}}, % in a neighbourhood of $m^\star$, 
the function $\phi$ is continuously differentiable in a neighborhood of $m$ 
and its derivative is given as
\begin{equation} \label{eq:derm}\nonumber
\phi'[m]  = \frac{d \phi[m]}{dm} =
- x[m]^\top J[m] z[m] < 0,
\end{equation}
where $x[m]$ is the eigenvector associated to the eigenvalue $\lambda[m]$ of $\Sym(D[m] A)$ and $z[m] = A x[m]$. 
\end{theorem}
\begin{proof}
According to \eqref{eq:Jk}, we indicate by $J = J[m]$ the set of indices of entries of $D = D[m]$
that are equal to $m$.
%\begin{equation}
%J_k = \left\{ \begin{array}{rl} 1 & \mbox{if} \ D_k = m \\ 0 & \mbox{otherwise} \end{array} \right.
%\end{equation} 
Under the given assumptions we have 
\begin{eqnarray*}
D'_k[m] & = &  1  \qquad \qquad \mbox{if} \ D_k[m] = m,
\\
D'_k[m] & = &  \beta'[m] \qquad \, \mbox{if} \ D_k[m] = \beta[m] \in (m,1),
\\
D'_k[m] & = &  0  \qquad \qquad \mbox{if} \ D_k[m] = 1.
\end{eqnarray*}
By Lemma \ref{lem:eigderiv} we have that
\begin{eqnarray}
\phi'[m] & = & -\lambda'[m] = -x[m]^\top D'[m] z[m] 
\nonumber
\\ 
& = & -\sum\limits_{k=1}^{n}  x_k[m] D'_k[m] z_k[m] = -\sum\limits_{k \in J}  D'_k[m] x_k[m] z_k[m].\nonumber
\end{eqnarray} 

If the index $k$ in the summation corresponds to a zero diagonal entry of $D'[m]$ then 
it does not contribute to it; similarly if $k$ is such that $D_k[m] = \beta[m] \in (m,1)$, 
by Theorem \ref{th:stru}, we have that $x_k[m] z_k[m] = 0$, hence the index does not contribute too.
In conclusion only indices corresponding to entries equal to $m$ enter into the summation
with value $1$. This implies that 
\begin{equation}\nonumber
\phi'[m] = -x[m]^\top D'[m] z[m] = -x[m]^\top J[m] z[m],
\end{equation}
which proves the statement.
\end{proof}

In view of Theorem~\ref{thm:phi-derivative}, applying Newton's method to the equation $\phi[m]=0$ yields the following iteration:
\begin{equation}\label{CNM1} 
m_{k+1} = m_{k}  + \frac{\phi[m_k]}{x[m_k]^\top J[m_k] z[m_k]},
\end{equation}
where the right-hand side can be computed knowing the extremizer $D[m_k]$ computed by the inner iteration in the $k$-th step.
For this, we expect that with a few iterations (on average $5$ in our experiments), 
we are able to obtain an approximation $m_k \approx m^\star$.

\subsection{Extension to multilayer networks}

The following result extends to the multilayer case an explicit and easily computable expression for the derivative of 
\[
\phi[m]= F(D_1[m],\ldots,D_k[m]) = -\lambda[m]
\]
with respect to $m$.

For $D=D_i$, we also denote as $J^{(i)}=J^{(i)}[m]$ the diagonal matrix with binary entries (either $0$ or $1$),
according to the following: 
\begin{equation} \label{eq:Jk-multilayer}\nonumber
J^{(i)}_{kk} = \left\{ \begin{array}{rl} 1 & \mbox{if} \ D_{kk} = m, \\[2mm] 0 & \mbox{otherwise}. \end{array} \right.
\end{equation} 

Similarly to Theorem \ref{thm:phi-derivative}, we have the following result which gives us a formula to implement the cheap Newton-type iteration discussed in Section \ref{sec:outer}.

\begin{theorem}[Derivative of $\phi$] 
\label{thm:phi-Mderivative}
Under Assumption~{\rm \ref{ass:D-m}} for all matrices $D_i[m]$, in a neighbourhood of $m^\star$, 
the function $\phi$ is continuously differentiable in a neighbourhood of $m$ 
and its derivative is given as
\begin{equation} \label{eq:derm_m} \nonumber
\phi'[m]  = \frac{d \phi[m]}{dm} =
- \sum\limits_{i=1}^{k} w_i[m]^\top J^{(i)}[m] z_i[m] < 0,
\end{equation}
where $w_i[m], z_i[m]$ are defined according to Section \rm \ref{sec:multilayers}. 
\end{theorem}

\subsection*{Strict contractivity}
%With the algorithm we have described, given a matrix $A$ such that $\mu_2(A)<0$, we compute the smallest $0<m\le1$, called $m^\star$, and a matrix $D[m^\star]\in\Omega_{m^\star}$ such that (at least locally) $\mu_2(D[m^\star]A) = \max_{D\in\Omega_{m^\star}} \mu_2(DA) = 0.$
In order to impose strict contractivity one can replace equation \eqref{eq:zero} with $F_{m} (D[m]) = -c$, $c>0$. In this case, we can extend the previous results easily.
%We can extend this easily, replacing $0$ in the previous by $-c$ for some positive $c$, which means 
%a strict contractivity condition with negative exponent $-c$. For this we solve the problem        
%\[
%\max_{D\in\Omega_{m^\star}} \mu_2(DA) = -c.
%\]
In particular, the upper bound for $m^\star$ in Theorem \ref{th:ubmstar} becomes
\[
m^\star \le 1 - \frac{|\mu(A)+c|}{\beta(A)},
\]
and the Newton iteration step \eqref{CNM1} takes the form
\[
m_{k+1} = m_{k}  + \frac{\phi[m_k] + c}{x[m_k]^\top J[m_k] z[m_k]}.
\]
Everything else remains exactly the same.

\subsection{Illustrative numerical examples} \label{sec:num_ex}

We consider first the simple $2 \times 2$  matrix
\begin{equation*} \label{ex:1}
A = \left( \begin{array}{rr} -2 & 1 \\ 2 & -3 \end{array} \right),
\end{equation*}
whose logarithmic $2$-norm is $\mu_2(A) = \lambda_{\max} \left( \left( A+A^\top \right)/2 \right) = -0.9189\ldots$ Applying Theorem \ref{th:ubmstar} and choosing the $2$-norm we get
\begin{equation*}
m^\star \le m^\star_{ub} = 0.7776\ldots
\end{equation*}
Indeed, the exact value can be computed exactly considering the  extremal matrices
\[
    D_1 = \left( \begin{array}{rr} 1 & 0 \\ 0 & m \end{array} \right), \qquad D_2 = \left( \begin{array}{rr} m & 0 \\ 0 & 1 \end{array} \right),
\]
and it turns out to be
$
m^\star = 10 - 4 \sqrt{6} = 0.202041\ldots
$
Applying the proposed numerical method, with a $4$-digit accuracy, we get indeed the extremizer $D_1$ with
$
m^\star \approx 0.2021,
$
which agrees very well with the exact value.

Next consider the $3 \times 3$ example,
\begin{equation} \label{ex:2}
A = \left( \begin{array}{rrr}     
    -2  &   1  &   2 \\
    -1  &  -3  &   1 \\
     0  &   4  &  -3 
		\end{array} \right),
\end{equation}
whose logarithmic $2$-norm is $\mu_2(A) = \lambda_{\max} \left( A+A^\top \right)/2 = -0.2058\ldots$.
Applying the proposed numerical method we get %(see Table \ref{tab:ex2})
$
m^\star \approx 0.8023,
$
which is obtained for the extremal diagonal matrix
\begin{equation}
D = \left( \begin{array}{rrr}
    1.0000  &       0  &       0 \\
         0  &  0.8023  &       0 \\
         0  &       0  &  1.0000
\end{array} \right),
\nonumber
\end{equation}
where we observe that all entries assume extremal values in the interval $[m^\star,1]$.
%\vskip -2cm
\begin{figure}[ht]
\centerline{
\includegraphics[scale=0.5]{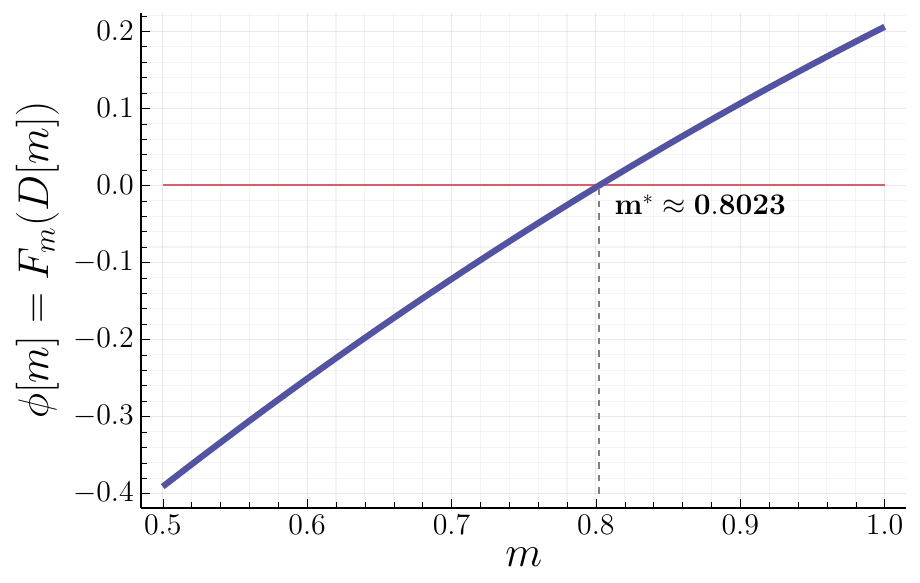} 
}
\vspace{-0.5cm}
\caption{The function $\phi[m]$ for the example matrix \eqref{ex:2}.  \label{fig:exN2}}
\end{figure}

Let us apply Newton's method to compute $m^\star$, using the costless 
derivative formula in Theorem \ref{thm:phi-derivative}.
The results are shown in Table \ref{tab:ex2}, which shows quadratic convergence.
\begin{table}[ht] 
\begin{center}
\begin{tabular}{c c c c} 
\toprule
$k$ & $m$ & $\phi[m]$ & $\phi'[m]$ \\ 
\midrule
$0$ & $0.9$ & $\ \ \,0.106308075414147$ &  $1.041560098368530$  \\
$1$ & $0.797933805662616$ & $-0.005021877535631$ &  $1.140880961822965$   \\
$2$ & $0.{\bf 8023}35559836519$ & $-9.673874699522855 \cdot 10^{-6} $ &  $1.136487060632446$      \\
$3$ & $0.{\bf 8023440719}21729$ & $-3.612360410798487 \cdot 10^{-11}$ &  $1.136478572827426$     \\
  \bottomrule
\end{tabular}
\caption{Example \ref{ex:2}: quadratic convergence of Newton's method. For $k=3$ we observe up to 10 exact digits, highlighted in bold font.}
\label{tab:ex2}
\end{center}
\end{table}

\section{A combinatorial relaxation} \label{sec:comb}

The numerical integration of \eqref{eq:probinn} is certainly a demanding step in terms of CPU time.
However, when slightly modifying the value $m$ (or considering a smooth function $A(t)$), in order to
solve \eqref{eq:probinn} for the new data, we propose an alternative strategy to the reintegration of \eqref{eq:probinn}, which is based on Theorem \ref{th:stru}.

% Note that the result has been formulated for $m^\star$ but indeed is immediately extended to any $m$.

\subsection{On the necessary condition \eqref{eq:condex} of Theorem \ref{th:stru}}

Theorem \ref{th:stru} provides an analytic condition to identify intermediate entries.
% ; it could be used in a numerical procedure aiming to solve Problem \eqref{eq:problem}.
%
Equation \eqref{eq:condex} is a codimension-$1$ condition. It means that for having entries of 
the extremizer $\Dm$  assuming intermediate values, either the $i$-th component of the 
eigenvector $x$ associated to the rightmost eigenvalue of $\Sym\left(\Dm A \right)$, or 
the $i$-th component of the vector $z=A x$ have to vanish. 

In the $2 \times 2$ case it is direct to prove that extremizers are either
\[
\diag(1,m^\star) \qquad \mbox{or} \qquad \diag(m^\star,1),
\]
and cannot assume intermediate values, as one expects. We omit the proof
for sake of conciseness.

\begin{comment} \rm

A consequence of Theorem \ref{th:stru} is that if 
\begin{itemize}
\item[(i) ] the logarithmic norm $\mu_2(A)$ is not large in modulus;
\item[(ii) ]  for the original matrix $\Sym(A)$ the vectors $x$ and $z=A x$ associated to the rightmost eigenvalue of
$\Sym(A)$ have entries suitably bounded away from zero;
\item[(iii) ] the interval $[m^\star,1]$ is relatively small; % w.r.t. to previous bounds, 
\end{itemize}
then it is unlikely that \eqref{eq:condex} is satisfied, in which case we would expect our problem to have a 
\emph{combinatorial structure}, with entries in $\{ m,1\}$ for $m \le m^\star$.

Even if we are not able to prove that the condition $x_i z_i \neq 0$ for all $i$ is nongeneric, we may argue it is unlikely. 
Proving non-genericity would be a very interesting result if this were to be true.

In general, we would expect that many entries assume extremal values and only a few --- if any --- would assume intermediate
values, which have to satisfy \eqref{eq:condex}.
%For a given $m \in [m^\star,1]$, extremal perturbations, which maximize $\mu_2(DA)$ for $D \in \Omega_m$ have entries which can assume
%only the values $m$ and $1$.

\end{comment}

In the general $n$-dimensional case, assume that we know an extremizer $\Dm$, whose diagonal entries only 
assume the extremal values $m$ and $1$ and that 
the eigenvector $x$ of $\Sym(\Dm A)$ and the vector $z=A x$ have entries suitably bounded away from zero.  
If we slightly perturb $m$ (or even the matrix $A$), then by continuity arguments we expect generically 
that the new extremizer  has the same entry-pattern of $\Dm$, which means that the new vectors $\widetilde x$ and $\widetilde z$ remain close to $x$ and $z$ and thus have still non-zero entries. This remains true also for
moderate changes of $m$.
As an example, for the matrix \eqref{ex:2} we get, 
 using the proposed inner-outer scheme:
\[
\left(x, x^\star \right) = \left( \begin{array}{rr}    
    0.3836  &  0.3595 \\
    0.6158  &  0.6664 \\
    0.6882  &  0.6532
\end{array} \right)
\qquad \mbox{and} \qquad
\left(z, z^\star \right) = \left( \begin{array}{rr}
    1.2251  &  1.2537 \\
   -1.5427  & -1.7055 \\
    0.3984  &  0.7059 \\
\end{array} \right),
\]
where $x$ indicates the rightmost eigenvector of $\Sym(A)$ and $z = A x$, $x^\star$
the rightmost eigenvector of $\Sym(D_{m^\star} A)$ and $z^\star = A x^\star$.
This illustrates the slight change between the above starting vectors and those in the
extremizer.
\begin{figure}[t]
\centerline{
\includegraphics[scale=0.4]{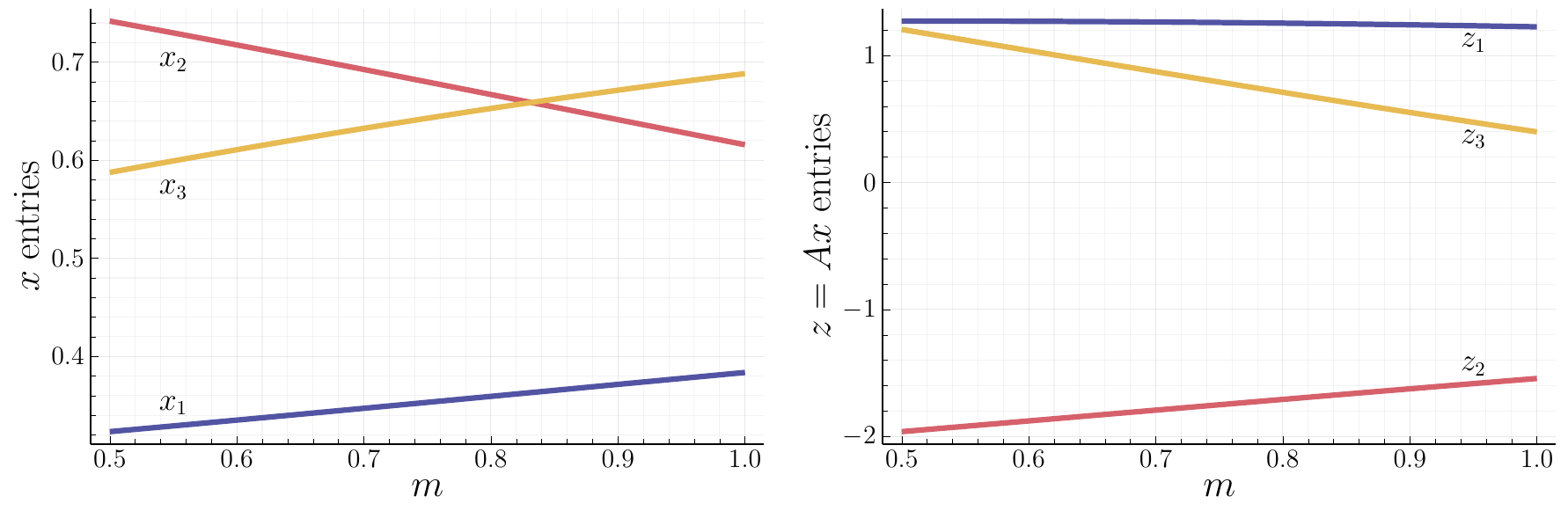}
}
%\vspace{-2cm}
\caption{The behaviour of the $3$ entries of $x$ (left) and $z=A x$ (right) for extremizers for problem \eqref{ex:2}
as a function of $m$.  \label{fig:ex2}}
\end{figure}
% Since both $x$ and $z$ have all components bounded away from zero, the property that
% all entries assume only the extremal values is predictable a priori.
The plot (as a function of $m \in [m^\star,1]$) of the entries of $x$ and of $z$ 
is shown in Figure~\ref{fig:ex2}.

\begin{remark} \rm
Extremizers in general are not unique. Consider for example the following matrix $A$ and $m=0.2$
\[
A = \left( \begin{array}{ccc}
 -3 & 1 & {3}/{2} \\
 -1 & -1 & 3 \\
 -1 & -3 & 0 \\
\end{array}
\right), \qquad \Sym(A) =
\left( \begin{array}{ccc}
 -3 &  0 & {1}/{4} \\
  0 & -1 & 0 \\
 {1}/{4} & 0 & 0 \\
\end{array}
\right).
\]
Then $D_1 = \diag\left( m,m,1 \right)$, $D_2 = \diag\left( 1,1,m \right)$ 
and $D_3 = \diag\left( m,1,m \right)$ are local extremizers, with $F(D_1) = -1.1427$, 
$F(D_2) = -0.8764$, $F(D_3) = -0.8237$, so that $D_1$ is the global minimizer.
Instead for the matrix $A$ in \eqref{ex:2} with $m=0.8$ the extremizer $D=\diag\left( 1,m,1 \right)$
turns out to be unique.
\end{remark}

\subsection{A greedy combinatorial algorithm} % Checking diagonal matrices with entries in $\{m,1\}$}
\label{sec:combinatorial-algorithm}

An algorithm to check possible extremizers only assuming extremal values is the following. Compute all possible
diagonal matrices with combinatorial structure (which are $2^n$) and compute the gradient $G$ for each of them. Looking at the sign of the vector $G$ we can assert whether a certain matrix $D$ is a local extremizer.
\begin{remark}[Efficient combinatorics] \rm
The combinatorial algorithm is very expensive when $n$ increases. However, starting from the knowledge
of an extremizer $\Dm$, the problem of minimizing $F$ for $m'$ close to $m$ might be efficiently
tackled by maintaining as an initial guess the extremizer $\Dm$, just replacing its entries equal to $m$
with the value $m'$, and then checking the sign of the gradient. Similarly one could tackle the problem
where the matrix $A$ instead of $m$ is slightly changed, as we will see in the next section. 
This would allow to avoid the numerical integration of the system of ODEs to minimize $F$
and is somehow reminiscent of an approach based on interior point methods as opposed to simplex 
methods in linear programming problems.
\end{remark}

For example \eqref{ex:2}, this gives only the detected matrix (for all $m \le m^\star$)
\begin{equation*} %\label{eq:ex2opt}
\Dm = \left( \begin{array}{rrr}
         1  &       0  &       0 \\
         0  &       m  &       0 \\
         0  &       0  &       1
\end{array} \right).
\end{equation*}
% This would become very expensive when the dimension of the problem increases (combinatorial explosion).
% Instead of generating all diagonal matrices, another possibility would be to start with an extremal matrix $D$ 
% and, according to the sign of the gradient, changing only those entries which would be moved by the gradient system towards  the interior.

For example, still considering Example \eqref{ex:2}, for the starting non-optimal diagonal matrix
\begin{equation} \nonumber
D = \diag\left( m,m,1 \right),
\end{equation}
with $m=0.9$, we get $\lambda = -0.1523$ and
\begin{equation} \nonumber
G = \diag\left( \mathbf{0.0700}, -0.1602, 0.0580 \right),
\end{equation}
which indicates that the optimality condition for the first entry is violated.
Changing the value of the first entry to the opposite extremal value and 
modifying to 
\begin{equation} \label{eq:Dhm}
\hat{D}[m] = \diag\left( 1,m,1 \right),
\end{equation}
allows us to find the optimal solution for the problem. %, that is the extremizer \eqref{eq:ex2opt}.
At this point, it is sufficient to tune $m$ through the scalar equation
\[
\mbox{find} \ m \ : \: \lambda_{\max} \left( \Sym(\hat{D}[m] A) \right) = 0, 
\]
where $\hat{D}[m]$ has the structure in \eqref{eq:Dhm}
to obtain the optimal value $m^\star \approx 0.8023$.

\section{Time-dependent case: problems (P1) and (P2)} \label{sec:time}

    We are now interested to consider the more general time-varying problem  setting
%We consider now the more general problem 
where $A = A(t)$, for which we assume
\begin{equation} \nonumber
\mu_2\left( A(t) \right) < 0.
\end{equation}
The problem now would be that of computing the function of $t$
\begin{equation} \label{eq:tdproblem}
    m^\star(t) = \argmin\limits_{m \in [0,1]} \, \{m \ : \  \mu_2\left (D A(t) \right) \le 0 \quad \mbox{for all matrices $D \in \Omega_m$} \}.
\end{equation}
We denote by $\Dm(t)$ an associated extremizer, that is such that $\mu_2\left( \Dm(t) A(t) \right) = 0$,
and by $m^\star(t)$ its smallest entry.
To a given diagonal matrix  $D \in \Omega_m$, we associate the structured matrix
\begin{equation} \label{eq:stru}\nonumber
S(D) = \{ I_1, I_2, I_3 \},
\end{equation} 
with
\begin{eqnarray} \nonumber
I_1 & = & \{ i \in \{1,\ldots,n\} : D_{ii} = 1 \},
\\
I_2 & = & \{ i \in \{1,\ldots,n\} : D_{ii} = m \},
\label{eq:I123}
\\
\nonumber
I_3 & = & \{ i \in \{1,\ldots,n\} : D_{ii} = \beta_i \in (m,1) \}.
\end{eqnarray}

By continuity arguments, we have generically that the structure of $S(\Dm(t))$ is piecewise constant
along $[0,T]$, that means that indexes of the entries equal to $1$ and to $m(t)$, as well as those assuming intermediate
values, remain so within certain intervals.

Let us consider the symmetric part of $\Dm(t) A(t)$. Applying Lemma \ref{lem:eigderiv} to the rightmost eigenvalue
of $\Sym(\Dm(t) A(t))$, under smoothness assumptions we obtain (where time dependence is omitted for conciseness):
\begin{eqnarray}
0 \equiv \frac{d}{dt} \lambda = 
\dot{\lambda} 
& = & \frac12 x^\top \left( \dot{D} A + A^\top \dot{D} + D \dot{A} + \dot{A}^\top D \right) x
\nonumber
\\
& = &
\frac12 x^\top \left( \dot{D} A + A^\top \dot{D} \right) x + \gamma
=
x^\top \dot{D} z + \gamma ,\nonumber
\label{eq:dert}
\end{eqnarray}
with $z = A x$ and 
\begin{equation} \label{eq:gamma}\nonumber
\gamma = x^\top \left( D \dot{A} + \dot{A}^\top D \right) x.
\end{equation}

Then we have (exploiting Theorem \ref{th:stru})
\begin{equation}\nonumber
\sum\limits_{i=1}^{n} x_i \dot{D}_{ii} z_i = 
\sum\limits_{i \in I_2} x_i \dot{D}_{ii} z_i  = 
\left( \sum\limits_{i \in I_2} x_i z_i \right) \dot{m}^\star(t). 
\end{equation}
Consequently, we obtain the following scalar ODE for $m^\star(t)$,
\begin{equation} \label{eq:odem}
\dot{m}^\star(t) = -\frac{\gamma(t)}{\zeta(t)}, \qquad \zeta(t) = \sum\limits_{i \in I_2} x_i(t) z_i(t).
\end{equation}
This shows that from the sign of the r.h.s.\ of \eqref{eq:odem} we obtain the growth/decrease of $m^\star(t)$.
The resulting algorithm for the time-varying case is discussed in the next subsection.

\subsection{Algorithm for the time varying problem}

On a uniform time grid $\{ t_k \}_{k \ge 0}$ ($t_{k+1}=t_k + \Delta t$), given an extremizer $D(t_k)$ with minimal 
entry $m^\star(t_k) = m^\star_k$, we first approximate $m^\star(t_{k+1})$ by an Euler step (where we denote as 
$A^{(k)} = A(t_k)$ and $\dot{A}^{(k)} = \dot{A}(t_k)$):
\begin{equation}\nonumber
m_{k+1} = m^\star_k - \Delta t \frac{\gamma_k}{\zeta_k}  ,
\end{equation}
with
\begin{equation*}
\gamma_k = (x^{(k)})^\top \left( D^{(k)} \dot{A}^{(k)} + (\dot{A}^{(k)})^\top D^{(k)} \right) x^{(k)}, \qquad
\zeta_k = \sum\limits_{i \in I_2} x^{(k)}_i(t) z^{(k)}_i(t),
\end{equation*}
with $x^{(k)}$ eigenvector of $\Sym\left( D^{(k)} A^{(k)} \right)$ associated to the rightmost eigenvalue and
$z^{(k)}= A x^{(k)}$.
Next, we update $D(t_{k+1})$ by setting its minimal entries as $m_{k+1}$ and compute 
$x^{(k+1)}$, i.e. the eigenvector of $\Sym\left( D^{(k+1)} A(t_{k+1}) \right)$ associated to its rightmost eigenvalue,
and set $z^{(k+1)}= A x^{(k+1)}$.
Then, in order to obtain a more accurate approximation of $m^\star_{k+1}$, we apply a single Newton step, as described in 
Section \ref{sec:outer}.
Finally we verify that the resultant matrix $D^{(k+1)} = D[m^\star_{k+1}]$ is an extremizer by checking 
(see \eqref{eq:I123})
\begin{equation}
 x^{(k+1)}_i z^{(k+1)}_i   > 0 \quad \mbox{for} \ i \in I_1 
% \label{eq:cond1}
\qquad 
\mbox{and}
\qquad
 x^{(k+1)}_i z^{(k+1)}_i   < 0 \quad \mbox{for} \ i \in I_2 . \label{eq:cond12}
\end{equation} 
If, for at least one index $i$, one of the previous conditions \eqref{eq:cond12} is violated,
then an optimizer with a new structure is searched by running the whole Algorithm described in Section
\ref{sec:meth} for the matrix $A=A(t_{k+1})$.

\subsection{Illustrative examples}

Consider next the matrix
\begin{equation} \label{ex:3}
A(t) = \left( \begin{array}{rr}
-2-\sin(t) & \cos(t) \\ 
2 & -2-\cos(t)
\end{array}
\right), \qquad t \in [0,T], \ T=\pi,
\end{equation}
which for $t=0$ coincides to Example \eqref{ex:1}.
Figure \ref{fig:exms1} illustrates the behaviour of $m^\star(t)$.
% \vskip -3cm
%
%
%
\begin{figure}[t]
     \centering
     \begin{subfigure}[b]{0.45\textwidth}
         \centering
         \includegraphics[width=0.99\columnwidth]{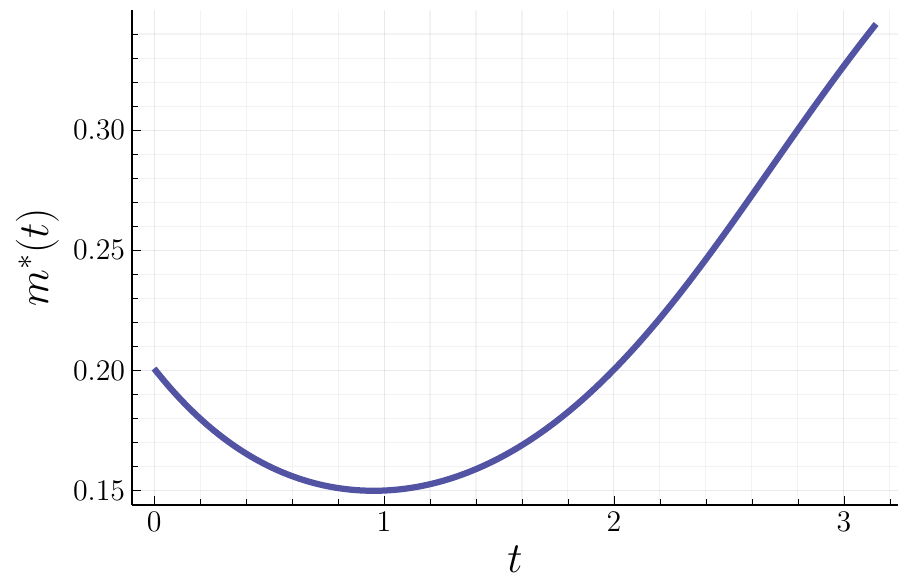}
         \caption{$A(t)$ as in example \eqref{ex:3}.} \label{fig:exms1}
     \end{subfigure}
     \hfill
     \begin{subfigure}[b]{0.45\textwidth}
         \centering
         \includegraphics[width=0.99\columnwidth]{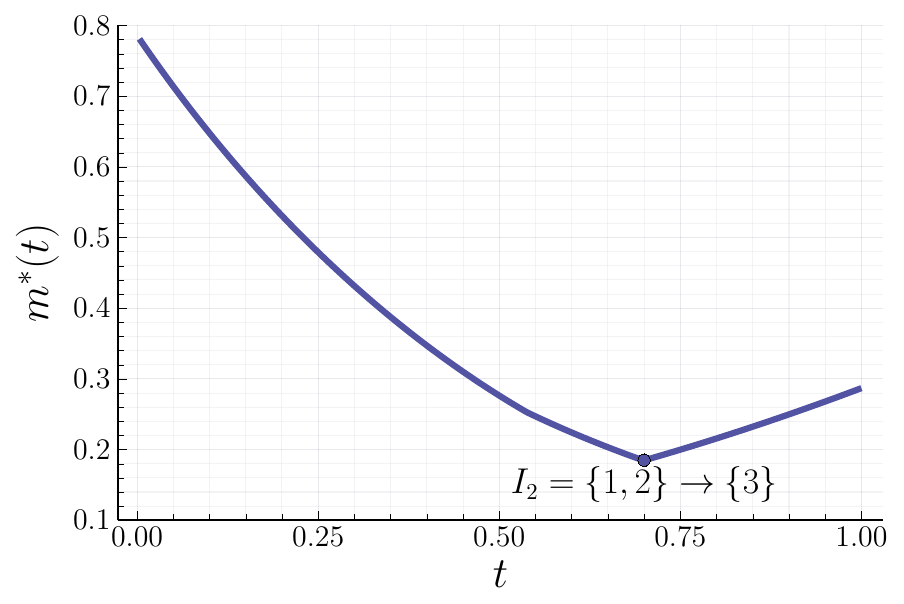}
         \caption{$A(t)$ as in example \eqref{ex:4}.}  \label{fig:exms2}
     \end{subfigure}
     \caption{Behaviour of $m^\star$ as a function of $t$, for two different matrix-valued functions $A(t)$.}
\end{figure}
%
% \begin{figure}[hbt]
% \centerline{
% \includegraphics[scale=0.43]{figmsex1.pdf}
% }
% \vspace{-4cm}
% \caption{The behaviour of the $m^\star$ for the matrix valued function \eqref{ex:3}
% as a function of $t$.  \label{fig:exms1}}
% \end{figure}
Next, consider the matrix
\begin{equation} \label{ex:4}
A(t) = \left(
\begin{array}{ccc}
 -t-1 & 1 & \frac{t}{2}+\frac{1}{2} \\
 -1 & t-3 & t+1 \\
 3-2 t & 1-2 t & 2 t-4 \\
\end{array}
\right),\qquad t \in [0,T], \ T=1.
\end{equation}
Figure \ref{fig:exms2} shows the sudden change of extremizer at $t \approx 0.7$.
The set $I_2 = \{ 1,2 \}$ switches to $I_2 = \{ 3 \}$, when $t$ crosses the discontinuity point for $(m^\star)'(t)$.
% \vskip -3cm
%
%
% \begin{figure}[hbt]
% \centerline{
% \includegraphics[scale=0.43]{figmsex2.pdf}
% }
% \vspace{-3cm}
% \caption{The behaviour of the $m^\star$ for the matrix valued function \eqref{ex:4}
% as a function of $t$.  \label{fig:exms2}}
% \end{figure}

\subsection{Remarks}

Applying the proposed method we obtain an accurate approximation of the extremizers $\Dm(t)$ for the considered
problem \eqref{eq:tdproblem} for $t=t_k,\ k=0,1,\ldots$
To have a uniform choice which guarantees contractivity we have to choose 
\begin{equation*}
m^\star = \max\limits_{k} \, m^\star_k,
\end{equation*}
which might be too restrictive.
Another possibility would be that of changing the slope $m = m(t)$ in the function $\sigma(t)$.
Finally, relaxing the bound
\begin{equation}\nonumber
\lambda_{\max} \left( \Sym \left (\Dm(t) A(t) \right) \right) \le 0 \qquad \mbox{to} \qquad 
\lambda_{\max} \left( \Sym \left (\Dm(t) A(t) \right) \right) \le \alpha 
\end{equation}
for some moderate $\alpha>0$, would allow maintaining control on the growth of the solution (the constant $C$ in \eqref{eq:contr} would be
$C \le \mathrm{e}^{\alpha T}$), without restricting the range of $\sigma'$.

% In this case a probabilistic analysis on the error bound \eqref{eq:contr} would be helpful.

% \subsection{Conclusions}

\subsection{A different outlook. Problem (P2): stability bound for fixed $m$} \label{sec:different}

Assume that $m$ is fixed and we want to compute a bound for the growth of the error in the 
solution of Problem (P2), that is to get a worst-case estimate of the type \eqref{eq:contr}. 
We write the solution to the differential equation ${\displaystyle {\dot {u}}=D(t) A(t)u}$, $u(t_0)=u_0$ (with $D(t)$ unknown a priori), as
\begin{equation} \nonumber
u(t) = \Phi(t,t_0) u_0.
\end{equation}
By Gr\"onwall Lemma, it can be shown that a bound for the norm of the state transition matrix 
${\displaystyle \Phi (t,t_{0})}$ is given by
\begin{equation}\nonumber
\| \Phi (t,t_{0}) \| \leq \exp \left(  \int\limits_{t_{0}}^{t} \mu_2 \left( D(s) A(s) \right) ds\right),
\end{equation}
for all $t \ge t_0$.
This means that if we compute the worst case quantity $\mu_2(s) = \mu_2(D(s) A(s))$ for $D(s) \in \Omega_m$, by applying
a quadrature formula we get an upper bound for the constant $C$, that is
\begin{equation} \label{eq:boundC}
C \le \exp \left(  \int\limits_{0}^{T} \mu_2 \left( \widetilde{D}(s) A(s) \right) ds\right) \approx 
\exp\left( Q \left[ \mu_2 \left( \widetilde{D}(s) A(s) \right) \right] \right) := \widetilde{C},
\end{equation}
with $\widetilde{D}(t) = \argmax\limits_{D \in \Omega_m} \mu_2\left( D A(t) \right)$
%\begin{equation} \label{eq:Dcrit}
%\widetilde{D}(t) = \argmax\limits_{D \in \Omega_m} \mu\left( D A(t) \right),
%\end{equation}
% \begin{equation} \label{eq:Dcrit}
% \widetilde{D}(t) \longrightarrow \max\limits_{D \in \Omega_m} \mu\left( D A(t) \right).
% \end{equation}
and
$Q[\cdot]$ a suitable quadrature formula, e.g. the trapezoidal rule.
% \vskip -1cm
\begin{figure}[t]
\centerline{
\includegraphics[scale=0.5]{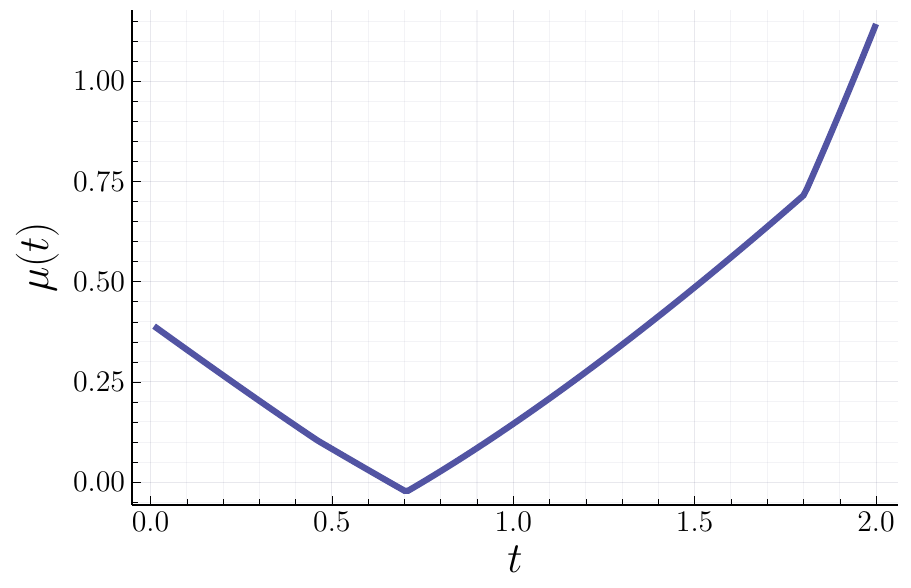}
}
%\vspace{-3cm}
\caption{The behaviour (for fixed $m$) of $\mu_2(t)$ (see \eqref{eq:mut})  for the matrix-valued function \eqref{ex:4}
as a function of $t$.  \label{fig:exms3}}
\end{figure}
Figure \ref{fig:exms3} illustrates the behaviour of 
\begin{equation} \label{eq:mut}
\mu_2(t) : =\mu_2 \left( \widetilde{D}(t) A(t) \right)
\end{equation}
for the matrix-valued function $A(t)$ of example \eqref{ex:4},  where for the quadrature we have used trapezoidal rule.
\begin{remark} \rm
It is not guaranteed that there exists a trajectory $u(t)$ of the dynamical system \eqref{eq:oderesnet} to which corresponds
the critical matrix sequence $\widetilde{D}(t)$. %  given by \eqref{eq:Dcrit}.
However, if such a trajectory exists, then the bound \eqref{eq:boundC} is sharp.
\end{remark}
\subsection{Illustrative example}
Let us consider the Cauchy problem \eqref{eq:oderesnet} with $T=2$ and the example matrix \eqref{ex:4}. 
Note that, for $m=0.5$, the problem is strictly contractive in the spectral norm. 
We compute the upper bound \eqref{eq:boundC} for several values of $m$. The results are illustrated in Table \ref{tab:3}.
\begin{table}[th] 
\begin{center}
\begin{tabular}{c c c c} 
\toprule
$m$ & $\min\limits_{t} \mu_2(t)$  & $Q[\mu_2(t)]$ & $\widetilde{C}$  \\ 
\midrule
$0.01$  & $0.2890$  & $1.2467$  &  $3.4787$  \\
$0.05$  & $0.2202$  & $1.1170$  &  $3.0556$  \\
$0.1 $  & $0.1379$  & $0.9605$  &  $2.6129$  \\
$0.2 $  & $-0.0216$ & $0.6610$  &  $1.9368$  \\  
$0.5 $  & $-0.4076$ & $-0.1414$ &  $0.8681$  \\
\bottomrule
\end{tabular}
\caption{The upper bound $\widetilde{C}$ for Example \eqref{ex:4}.}
\label{tab:3}
\end{center}
\end{table}

\section{An application: a neural network for digits classification} \label{sec:final_ex}

To conclude, we apply the theory developed in the previous sections to increase the robustness of a neural ODE image classifier. To this end, we consider MNIST handwritten digits dataset \cite{deng2012mnist} perturbed via the Fast Gradient Sign Method (FGSM) adversarial attack \cite{goodfellow2014explaining}.  

Here, in order to maintain accuracy (the percentage of correctly classified testing images) of the neural network, the minimal slope $m<1$ cannot be too large. For this reason we have to face Problem (P3), since when we solve Problem (P1) with the learnt matrix $A$, it can happen that the value $m^*$ computed by solving Problem (P1) in order to get uniform contractivity of the neural ODE, might be larger than $m$. In this case we modify the matrix $A$ in order to increase the stability of the network, by simple perturbations of the form $- \delta \Id$, that is suitably shifting $A$ by a scaled identity matrix. The use of optimal modifications of $A$ is not considered in this article, although we consider it an important research topic.

Let us recall that MNIST consists of 70000  $28\times28$ grayscale images (60000 training images and 10000 testing images), that is vectors of length 784 after vectorization. We consider a neural network made up of the following blocks:
\begin{itemize}
    \item[(a)] a downsampling affine layer that reduces the dimension of the input from 784 to 64, i.e. a simple transformation of the kind $y = A_1x + b_1$,
    where $x\in\R^{784}$ is the input, $y\in\R^{64}$ is the output, and $A_1\in\R^{64,784}$ and $b_1\in\R^{64}$ are the parameters;
    \item[(b)] a neural ODE block that models the feature propagation,
    \[
        \begin{cases}
            \dot{x}(t) = \sigma\left(Ax(t)+b\right), \qquad t\in[0,1],\\
            x(0) = y,
        \end{cases}
    \]
     whose initial value is the output of the previous layer,
    where $x(t):[0,1]\to\R^{64}$ is the feature vector evolution function, $A\in\R^{64,64}$ and $b\in\R^{64}$ are the parameters, and $\sigma$ is a custom activation function, defined as
    {
    \[
        \sigma(x) =
        \begin{cases}
            x, \qquad &\mbox{if } x \geq 0,\\
            \tanh{x}, \qquad &\mbox{if }  -\bar{x} \le x < 0,\\
             \alpha x, \qquad &\mbox{otherwise},
        \end{cases}
    \]
    where $\bar{x}>0$ is such that $\tanh'{(\pm\bar{x})}=\alpha=0.1$ (see Figure \ref{fig:custom_act_fun});
    }
    \begin{figure}[h]
        \centering
        \includegraphics[scale=0.5]{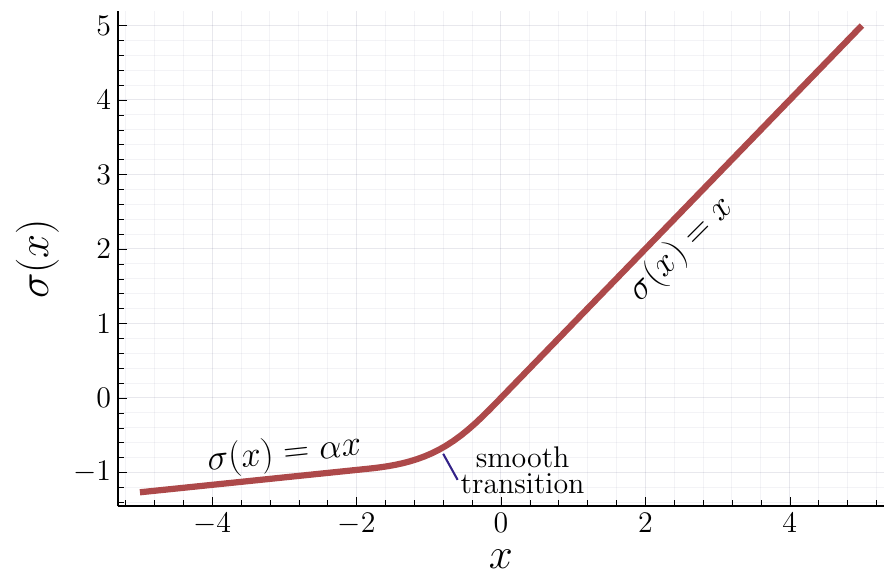}
        \caption{Custom activation function: a smoothed Leaky Rectified Linear Unit (LeakyReLU) with minimal slope $\alpha=0.1$.}
        \label{fig:custom_act_fun}
    \end{figure}
    \item[(c)] a final classification layer that reduces the dimension of the input from 64 to 10, 
    % (the number of classes in MNIST), 
    followed by the softmax activation function
    \[
        x_{out} = \mbox{softmax}\left(A_2x(1)+b_2\right), 
    \]
    where $x(1)\in\R^{64}$ is the output of the neural ODE block, $A_2\in\R^{10,64}$ and $b_2\in\R^{10}$ are the parameters, and $x_{out}$ is the output vector whose component $i$ is the probability that the input $x$ belongs to the class $i$. 
        Recall that $\mbox{softmax}$ is a vector-valued function that maps the vector $x$ into the vector $\mbox{softmax}(x)=e^x\|e^x\|_1^{-1}$, where exponentiation is done entrywise.

    %  Define softmax 
\end{itemize}

We require the neural ODE block to be contractive, so that small perturbations added to the input, such as adversarial attacks, are not amplified.

Making use of the proposed algorithm for solving Problem (P1), we compute the smallest value $m^\star$ that the entries of the diagonal matrix $D$ can assume, such that
\begin{equation}\label{eq:xyz}
     \sup\left\{ \mu_2(DA),\ \forall D\in\D^{64,64} \mbox{ such that } D_{ii}\in[m^\star,1] \right\} = 0,
\end{equation}
which implies the neural ODE is contractive if $m^\star\le\alpha$. If the value $m^\star$ computed by the algorithm in Section \ref{sec:outer} is larger than $\alpha$, 
we set $\delta > 0$ a small positive constant and replace the matrix $A$ by the shifted matrix
\[
    \hat{A} = A - \ell \delta I,
\]
where $\ell$ is the smallest positive integer such that \eqref{eq:xyz} holds for $\hat{A}$ in place of $A$.

Note that shifting the matrix $A$, and thus its entire spectrum, to get $m^\star$ below the specified threshold
$\alpha$, is a greedy algorithm that allows us to obtain \eqref{eq:xyz}. Other strategies may be used here, including computing the nearest matrix to $A$ in some suitable metric while ensuring $m^\star \le \alpha$. 

We add the proposed shifting to the standard training methodology via adjoint method \cite{chen2018neural} after parameter initialization and after each step of gradient descent to guarantee that \eqref{eq:xyz} holds with $m^\star\le\alpha$ during training.
\begin{remark} \rm 
    Observe that we apply the algorithm in Section \ref{sec:outer} to compute $m^\star$ after parameter initialization and after each step of gradient descent during training, to decide if the weight matrix needs to be shifted or not. Therefore, the algorithm for solving Problem (P1) is applied each time to a different matrix.
\end{remark}
Then, we train a first version of the above-mentioned neural network according to standard training, and we train a second version of the same neural network employing the proposed modification to the standard training. The networks are trained for 70 epochs. As a further comparison, we consider two models in the literature to improve adversarial robustness: the structural approach \texttt{asymODEnet}, where the vector field is constrained to be antisymmetric, i.e.
\[
    \dot{x}(t) = - A^\top \sigma(A x(t) + b), \quad t\in[0,1],
\]
considered for example in \cite{ruthotto, celledoni2023dynamical, sherry2024designing}, and the method \texttt{SODEF} by Kang et al. in \cite{kang2021stable}, which is a neural ODE whose equilibrium points correspond to the classes of the considered dataset. By ensuring that the equilibrium points are Lyapunov-stable, the solution for an input with a small perturbation is guaranteed to converge to the same solution corresponding to the unperturbed initial datum.

Loss and accuracy values for all four aforementioned methods are provided in Figure~\ref{fig:losses}. We eventually compare in Table \ref{tab:tacomp1} the test accuracy of the considered models as a function of $\varepsilon>0$. This quantity measures the size in the $L^{\infty}$-norm of the FGSM perturbation applied to each testing image, i.e. if $x_i$ is a testing image, $x_i+\varepsilon\delta_i$ is the perturbed image, with $\delta_i$ the FGSM attack on $x_i$ and $\|\delta_i\|_\infty=1$. Incorporating the strategy proposed here, we observe a clear improvement in robustness against the FGSM attack.

%Then, for each image $x_i$ in the test set, we compute the accuracy after a tailored adversarial perturbation $x_i+ \varepsilon\delta_i$, where $\delta_i$ is computed via the FGSM with $\|\delta_i\|=1$ and $\varepsilon>0$ is a parameter we use to tune the size of the perturbation.
    
%\begin{figure}[t]
%    \centering
%    \includegraphics[scale=0.5]{F6.pdf}
%    \caption{Comparison of test accuracy as a function of $\varepsilon$. }
%    \label{fig:tacomp}
%\end{figure}
%Figure \ref{fig:tacomp} shows the test accuracy of two neural networks  as a function of $\varepsilon$, the former trained according to standard training and the latter trained using the modification deriving from the algorithm presented in this work.
%The networks are trained for 70 epochs, and $c$ is chosen to be 0. Incorporating the tuning and shifting strategy proposed here, we observe a clear improvement in robustness against the FGSM attack.
%In table \ref{tab:tacomp} we report some exact values.

\begin{table}[ht]
    \centering
    \begin{tabular}{ c  c  c  c  c  c c  c }
    \toprule
     $\varepsilon$       & 0         & 0.01       & 0.02      & 0.03     & 0.04      & 0.05      & 0.06       \\
    \midrule
     \texttt{ODEnet} \cite{chen2018neural}      & \textbf{0.9707}         & 0.9447       & 0.8927      & 0.7974     & 0.6559      & 0.4894      & 0.343       \\
     \texttt{stabODEnet}       &  0.9684 & \textbf{0.9496} & \textbf{0.9188} &  \textbf{0.8743} &  \textbf{0.8049} &  \textbf{0.7019} &  \textbf{0.5769} \\
     \texttt{SODEF} \cite{kang2021stable}      & \textbf{0.9707} & 0.9438 & 0.8932 & 0.8001 & 0.6721 & 0.522 & 0.3893    \\
     \texttt{asymODEnet} \cite{ruthotto}      & 0.9321 & 0.8898 & 0.8196 & 0.7439 & 0.6381 & 0.4968 & 0.3565 \\
    \bottomrule
    \end{tabular}
    \caption{Comparison of test accuracy across different models as a function of $\varepsilon$. Tests are implemented fixing the Euler step to $h=0.05$. The models with the best performance are highlighted.}
    \label{tab:tacomp1}
\end{table}

\begin{figure}[htbp]
    \centering
    \includegraphics[width = \columnwidth]{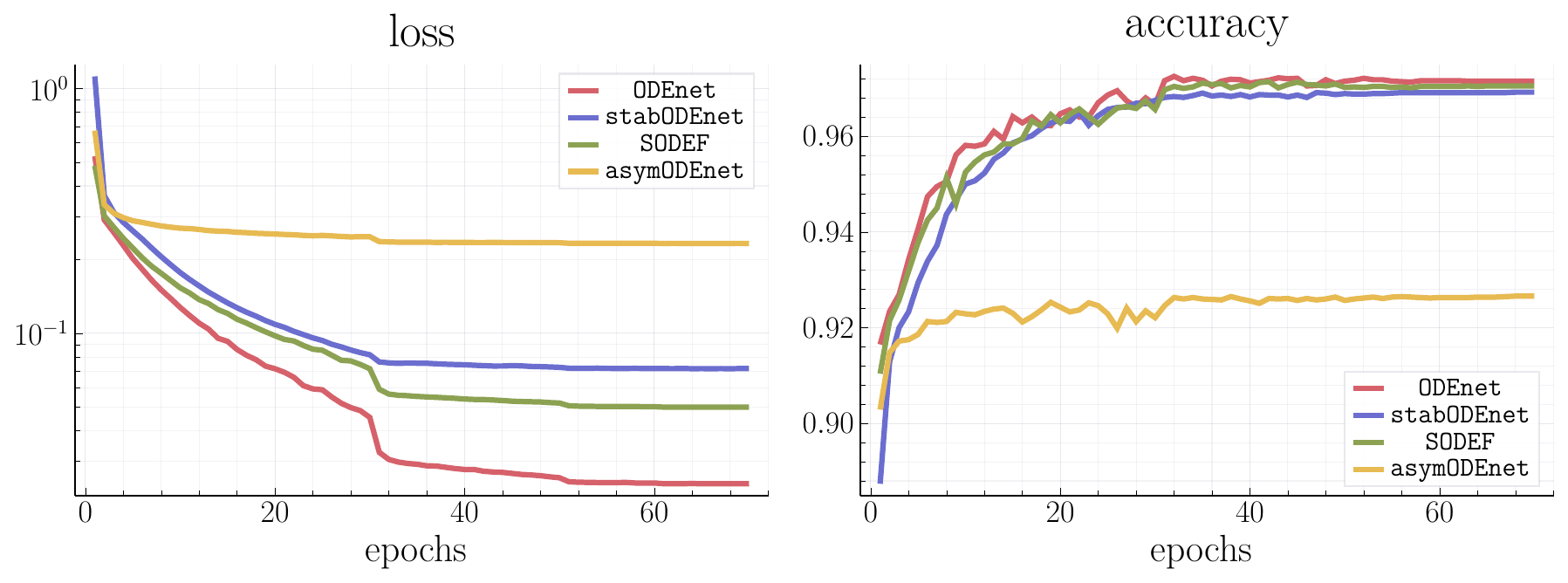}
    \caption{
    Losses (on the left) and accuracy (on the right) behaviour during training for \texttt{ODEnet}, stabilized \texttt{stabODEnet}, \texttt{SODEF}, and \texttt{asymODEnet}.
    \label{fig:losses}
    }
\end{figure}

    \begin{remark} \rm
        In our experiments, Table~\ref{tab:tacomp1}, the neural ODE block (b) has been numerically integrated with the forward Euler method with step size $h=0.05$. Notice that, if the step size $h$ is sufficiently small and the neural ODE is contractive, then the discretization of the neural ODE preserves the contractivity of its continuous counterpart. Table~\ref{tab:hs} shows the performance of the model with the neural ODE block discretized using different step sizes. It is worth pointing out that the contractive behaviour seems to decrease approximately when $h=0.33$.

    \end{remark}

    \begin{table}
      \centering
      \begin{tabular}{ c  c  c  c  c  c  c  c  }
         \toprule
       $\varepsilon$       & 0         & 0.01       & 0.02      & 0.03     & 0.04      & 0.05      & 0.06    \\
       \midrule
        \( h = 0.01\) &  \textbf{0.9699} & \textbf{0.9508} & 0.9260 & 0.8826 & 0.8116 & 0.7163 & 0.5880 \\
        \( h = 0.0125\) &  0.9697 & 0.9495 & 0.9235 & 0.8825 & 0.8151 & 0.7214 & 0.6014 \\
       \( h = 0.025\) &  0.9685 & 0.9503 & 0.9236 & 0.8804 & 0.8133 & 0.7219 & 0.6033 \\
        \( h = 0.05\)       &  0.9684 & 0.9496 & 0.9188 &  0.8743 &  0.8049 &  0.7019 &  0.5769 \\
       \( h = 0.1\) & 0.9689 & 0.9502 & \textbf{0.9262} & \textbf{0.8868} & 0.8229 & 0.7358 & 0.6261 \\
       \( h = 0.2\) & 0.9656 & 0.9456 & 0.9192  & 0.8807 & \textbf{0.8234} & \textbf{0.7381}  & \textbf{0.6318} \\
        \( h = 0.33\) & 0.9472 & 0.9288 & 0.9017  & 0.8641 & 0.8145 & 0.7429  & 0.6354 \\
      \bottomrule
      \end{tabular}
      \caption{Comparison of test accuracy as a function of $\varepsilon$ for the varying Euler step \( h \) for the stabilized \texttt{stabODEnet} (our method).  The settings with the best performance are highlighted.}
      \label{tab:hs}
  \end{table}

\section{Conclusions} \label{sec:conc}

In this article we have proposed a novel methodology to solve numerically a class of eigenvalue optimization problems arising when analyzing contractivity of neural ODEs. Solving such problems is essential in order to make a neural ODE stable (contractive) with respect to perturbations of initial data, and thus preventing the output of the neural ODE to be sensitive to small input perturbations, such those induced by adversarial attacks.

The proposed method consists of a two-level nested algorithm, consisting of an inner iteration for the solution of an inner eigenvalue optimization problem by a gradient system approach, and in an outer iteration tuning a parameter of the problem until a certain scalar equation is verified.

The proposed algorithm appears to be fast and its behaviour has been analyzed through several numerical examples, showing its effectiveness and reliability. A real application of a neural ODE to digits classification shows how this algorithm can be easily embedded in the contractive training of a neural ODE, making it robust and stable against adversarial attacks.

\subsection*{Acknowledgments}

We wish to thank an anonymous Referee for their helpful suggestions and constructive remarks.

N.G.\ acknowledges that his research was supported by funds from the Italian 
MUR (Ministero dell'Universit\`a e della Ricerca) within the PRIN 2022 Project ``Advanced numerical methods for time dependent parametric partial 
differential equations with applications'' and the PRO3 joint project entitled
``Calcolo scientifico per le scienze naturali, sociali e applicazioni: sviluppo metodologico e tecnologico''.  
N.G. and F.T. acknowledge support from MUR-PRO3 grant STANDS and PRIN-PNRR grant FIN4GEO.
%and
%the PRIN 2022 Project ``Advanced numerical methods for time dependent parametric partial 
%differential equations with applications''.

The authors are members of the INdAM-GNCS (Gruppo Nazionale di Calcolo Scientifico).

%\bibliographystyle{unsrt}
%\bibliography{NNbiblio}

\end{document}